\documentclass[reqno]{amsart}
\usepackage{amsfonts}

\newtheorem{theorem}{Theorem}
\theoremstyle{plain}

\newtheorem{corollary}{Corollary}

\newtheorem{definition}{Definition}

\newtheorem{lemma}{Lemma}

\newtheorem{proposition}{Proposition}

\numberwithin{equation}{section}
\numberwithin{theorem}{section}
\numberwithin{algorithm}{section}
\numberwithin{axiom}{section}
\numberwithin{case}{section}
\numberwithin{claim}{section}
\numberwithin{conclusion}{section}
\numberwithin{condition}{section}
\numberwithin{conjecture}{section}
\numberwithin{corollary}{section}
\numberwithin{criterion}{section}
\numberwithin{definition}{section}
\numberwithin{example}{section}
\numberwithin{exercise}{section}
\numberwithin{lemma}{section}
\numberwithin{notation}{section}
\numberwithin{problem}{section}
\numberwithin{proposition}{section}
\numberwithin{remark}{section}
\numberwithin{solution}{section}

\input{tcilatex}

\begin{document}
\title[Concentration compactness principle in critical dimension]{A remark
on the concentration compactness principle in critical dimension}
\author{Fengbo Hang}
\address{Courant Institute, New York University, 251 Mercer Street, New York
NY 10012}
\email{fengbo@cims.nyu.edu}

\begin{abstract}
We prove some refinements of concentration compactness principle for Sobolev
space $W^{1,n}$ on a smooth compact Riemannian manifold of dimension $n$. As
an application, we extend Aubin's theorem for functions on $\mathbb{S}^{n}$
with zero first order moments of the area element to higher order moments
case. Our arguments are very flexible and can be easily modified for
functions satisfying various boundary conditions or belonging to higher
order Sobolev spaces.
\end{abstract}

\maketitle

\section{Introduction\label{sec1}}

Let $n\geq 2$, $\Omega \subset \mathbb{R}^{n}$ be a bounded open subset with
smooth boundary. In \cite{M}, it is showed that for any $u\in
W_{0}^{1,n}\left( \Omega \right) \backslash \left\{ 0\right\} $,%
\begin{equation}
\int_{\Omega }\exp \left( a_{n}\frac{\left\vert u\right\vert ^{\frac{n}{n-1}}%
}{\left\Vert \nabla u\right\Vert _{L^{n}\left( \Omega \right) }^{\frac{n}{n-1%
}}}\right) dx\leq c\left( n\right) \left\vert \Omega \right\vert .
\label{eq1.1}
\end{equation}%
Here $\left\vert \Omega \right\vert $ is the volume of $\Omega $ and%
\begin{equation}
a_{n}=n\left\vert \mathbb{S}^{n-1}\right\vert ^{\frac{1}{n-1}}.
\label{eq1.2}
\end{equation}%
$\left\vert \mathbb{S}^{n-1}\right\vert $ is the volume of $\mathbb{S}^{n-1}$
under the standard metric. For convenience, we will use this notation $a_{n}$
throughout the paper. (\ref{eq1.1}) can be viewed as a limit case of the
Sobolev embedding theorem.

To study extremal problems related to (\ref{eq1.1}), a concentration
compactness theorem \cite[Theorem 1.6 on p196, Remark 1.18 on p199]{Ln} was
proved. For $n=2$, the argument is elegant. The approach is recently applied
in \cite{CH} on smooth Riemann surfaces to deduce refinements of the
concentration compactness principle (see \cite[Section 2]{CH}). These
refinements are crucial in extending Aubin's classical theorem on $\mathbb{S}%
^{2}$ for functions with zero first order moments of the area element (see 
\cite[Corollary 2 on p159]{A}) to higher order moments cases, motivated from
similar inequalities on $\mathbb{S}^{1}$ (see \cite{CH, GrS, OPS, W}).

For $n\geq 3$, due to the subtle analytical difference between weak
convergence in $L^{2}$ and $L^{p}$, $p\neq 2$, \cite[p197]{Ln} has to use
special symmetrization process to gain the pointwise convergence of the
gradient of functions considered. Unfortunately, as pointed out in \cite{CCH}%
, this argument is not sufficient to derive \cite[Remark 1.18]{Ln}. More
accurate argument is presented in \cite{CCH}. More recently, \cite{LLZ}
deduce similar results on domains in Heisenberg group without using
symmetrization process.

The main aim of this note is to extend the analysis in \cite{CH} from
dimension $2$ to dimensions at least $3$. In particular we will prove
refinements of concentration compactness principle (Proposition \ref{prop2.1}
and Theorem \ref{thm2.1}). Our approach also does not use symmetrization
process and is more close to \cite{CH, Ln}. It can be easily modified for
functions satisfying various boundary conditions or belonging to higher
order Sobolev spaces.

Throughout the paper, we will assume $\left( M^{n},g\right) $ is a smooth
compact Riemannian manifold with dimension $n\geq 2$. For an integrable
function $u$ on $M$, we denote%
\begin{equation}
\overline{u}=\frac{1}{\mu \left( M\right) }\int_{M}ud\mu .  \label{eq1.3}
\end{equation}%
Here $\mu $ is the measure associated with the Riemannian metric $g$. The
Moser-Trudinger inequality (see \cite{F}) tells us that for every $u\in
W^{1,n}\left( M\right) \backslash \left\{ 0\right\} $ with $\overline{u}=0$,
we have%
\begin{equation}
\int_{M}\exp \left( a_{n}\frac{\left\vert u\right\vert ^{\frac{n}{n-1}}}{%
\left\Vert \nabla u\right\Vert _{L^{n}\left( M\right) }^{\frac{n}{n-1}}}%
\right) d\mu \leq c\left( M,g\right) .  \label{eq1.4}
\end{equation}%
Here $a_{n}$ is given in (\ref{eq1.2}).

It follows from (\ref{eq1.4}) and Young's inequality that for any $u\in
W^{1,n}\left( M\right) $ with $\overline{u}=0$, we have the
Moser-Trudinger-Onofri inequality%
\begin{equation}
\log \int_{M}e^{nu}d\mu \leq \alpha _{n}\left\Vert \nabla u\right\Vert
_{L^{n}}^{n}+c_{1}\left( M,g\right) .  \label{eq1.5}
\end{equation}%
Here%
\begin{equation}
\alpha _{n}=\left( \frac{n-1}{n}\right) ^{n-1}\frac{1}{\left\vert \mathbb{S}%
^{n-1}\right\vert }.  \label{eq1.6}
\end{equation}%
We will use this notation $\alpha _{n}$ throughout the paper.

In Section \ref{sec2} we will derive some refinements of concentration
compactness principle in dimension $n$. These refinements will be used in
Section \ref{sec3} to extend Aubin's theorem on $\mathbb{S}^{n}$ to
vanishing higher order moments case. In Section \ref{sec4}, we discuss
modifications of our approach applied to higher order Sobolev spaces and
Sobolev spaces on surfaces with nonempty boundary.

Last but not least, we would like to thank the referee for careful reading
of the manuscript and many valuable suggestions.

\section{Concentration compactness principle in critical dimension\label%
{sec2}}

As usual, for $u\in W^{1,n}\left( M\right) $, we denote%
\begin{equation*}
\left\Vert u\right\Vert _{W^{1,n}\left( M\right) }=\left( \left\Vert
u\right\Vert _{L^{n}\left( M\right) }^{n}+\left\Vert \nabla u\right\Vert
_{L^{n}\left( M\right) }^{n}\right) ^{\frac{1}{n}}.
\end{equation*}%
We start with a basic consequence of Moser-Trudinger inequality (\ref{eq1.4}%
). It should be compared with \cite[Lemma 2.1]{CH}.

\begin{lemma}
\label{lem2.1}For any $u\in W^{1,n}\left( M\right) $ and $a>0$, we have%
\begin{equation}
\int_{M}e^{a\left\vert u\right\vert ^{\frac{n}{n-1}}}d\mu <\infty .
\label{eq2.1}
\end{equation}
\end{lemma}

\begin{proof}
We could use the same argument as in the proof of \cite[Lemma 2.1]{CH}.
Instead we modify the approach a little bit so that it also works for higher
order Sobolev spaces. Without losing of generality, we can assume $u$ is
unbounded. Let $\varepsilon >0$ be a tiny number to be determined, we can
find $v\in C^{\infty }\left( M\right) $ such that%
\begin{equation*}
\left\Vert u-v\right\Vert _{W^{1,n}\left( M\right) }<\varepsilon .
\end{equation*}%
We denote $w=u-v$, then%
\begin{equation*}
u=v+w=v+\overline{w}+w-\overline{w}.
\end{equation*}%
It follows that%
\begin{equation*}
\left\vert u\right\vert \leq \left\Vert v\right\Vert _{L^{\infty
}}+\left\vert \overline{w}\right\vert +\left\vert w-\overline{w}\right\vert .
\end{equation*}%
Hence%
\begin{equation*}
\left\vert u\right\vert ^{\frac{n}{n-1}}\leq 2^{\frac{1}{n-1}}\left(
\left\Vert v\right\Vert _{L^{\infty }}+\left\vert \overline{w}\right\vert
\right) ^{\frac{n}{n-1}}+2^{\frac{1}{n-1}}\left\vert w-\overline{w}%
\right\vert ^{\frac{n}{n-1}}.
\end{equation*}%
We get%
\begin{eqnarray*}
e^{a\left\vert u\right\vert ^{\frac{n}{n-1}}} &\leq &e^{2^{\frac{1}{n-1}%
}a\left( \left\Vert v\right\Vert _{L^{\infty }}+\left\vert \overline{w}%
\right\vert \right) ^{\frac{n}{n-1}}}e^{2^{\frac{1}{n-1}}a\left\vert w-%
\overline{w}\right\vert ^{\frac{n}{n-1}}} \\
&\leq &e^{2^{\frac{1}{n-1}}a\left( \left\Vert v\right\Vert _{L^{\infty
}}+\left\vert \overline{w}\right\vert \right) ^{\frac{n}{n-1}}}e^{a_{n}\frac{%
\left\vert w-\overline{w}\right\vert ^{\frac{n}{n-1}}}{\left\Vert \nabla
w\right\Vert _{L^{n}}^{\frac{n}{n-1}}}}
\end{eqnarray*}%
when $\varepsilon $ is small enough. It follows from (\ref{eq1.4}) that%
\begin{equation*}
\int_{M}e^{a\left\vert u\right\vert ^{\frac{n}{n-1}}}d\mu \leq c\left(
M,g\right) e^{2^{\frac{1}{n-1}}a\left( \left\Vert v\right\Vert _{L^{\infty
}}+\left\vert \overline{w}\right\vert \right) ^{\frac{n}{n-1}}}<\infty .
\end{equation*}
\end{proof}

Now we are ready to prove a localized version of \cite[Theorem 1.6]{Ln}. The
reader should compare it with \cite[Lemma 2.2]{CH}.

\begin{proposition}
\label{prop2.1}Assume $u_{i}\in W^{1,n}\left( M^{n}\right) $ such that $%
\overline{u_{i}}=0$, $u_{i}\rightharpoonup u$ weakly in $W^{1,n}\left(
M\right) $ and%
\begin{equation}
\left\vert \nabla u_{i}\right\vert ^{n}d\mu \rightarrow \left\vert \nabla
u\right\vert ^{n}d\mu +\sigma  \label{eq2.2}
\end{equation}%
as measure. If $x\in M$ and $p\in \mathbb{R}$ such that $0<p<\sigma \left(
\left\{ x\right\} \right) ^{-\frac{1}{n-1}}$, then for some $r>0$,%
\begin{equation}
\sup_{i}\int_{B_{r}\left( x\right) }e^{a_{n}p\left\vert u_{i}\right\vert ^{%
\frac{n}{n-1}}}d\mu <\infty .  \label{eq2.3}
\end{equation}%
Here%
\begin{equation}
a_{n}=n\left\vert \mathbb{S}^{n-1}\right\vert ^{\frac{1}{n-1}}.
\label{eq2.4}
\end{equation}
\end{proposition}

\begin{proof}
Fix $p_{1}\in \left( p,\sigma \left( \left\{ x\right\} \right) ^{-\frac{1}{%
n-1}}\right) $, then%
\begin{equation}
\sigma \left( \left\{ x\right\} \right) <\frac{1}{p_{1}^{n-1}}.
\label{eq2.5}
\end{equation}%
We can find a $\varepsilon >0$ such that%
\begin{equation}
\left( 1+\varepsilon \right) \sigma \left( \left\{ x\right\} \right) <\frac{1%
}{p_{1}^{n-1}}  \label{eq2.6}
\end{equation}%
and%
\begin{equation}
\left( 1+\varepsilon \right) p<p_{1}.  \label{eq2.7}
\end{equation}%
Let $v_{i}=u_{i}-u$, then $v_{i}\rightharpoonup 0$ weakly in $W^{1,n}\left(
M\right) $, $v_{i}\rightarrow 0$ in $L^{n}\left( M\right) $. For any $%
\varphi \in C^{\infty }\left( M\right) $, we have%
\begin{eqnarray*}
&&\left\Vert \nabla \left( \varphi v_{i}\right) \right\Vert _{L^{n}}^{n} \\
&\leq &\left( \left\Vert \varphi \nabla v_{i}\right\Vert _{L^{n}}+\left\Vert
v_{i}\nabla \varphi \right\Vert _{L^{n}}\right) ^{n} \\
&\leq &\left( \left\Vert \varphi \nabla u_{i}\right\Vert _{L^{n}}+\left\Vert
\varphi \nabla u\right\Vert _{L^{n}}+\left\Vert v_{i}\nabla \varphi
\right\Vert _{L^{n}}\right) ^{n} \\
&\leq &\left( 1+\varepsilon \right) \left\Vert \varphi \nabla
u_{i}\right\Vert _{L^{n}}^{n}+c\left( n,\varepsilon \right) \left\Vert
\varphi \nabla u\right\Vert _{L^{n}}^{n}+c\left( n,\varepsilon \right)
\left\Vert v_{i}\nabla \varphi \right\Vert _{L^{n}}^{n}.
\end{eqnarray*}%
It follows that%
\begin{eqnarray*}
&&\lim \sup_{i\rightarrow \infty }\left\Vert \nabla \left( \varphi
v_{i}\right) \right\Vert _{L^{n}}^{n} \\
&\leq &\left( 1+\varepsilon \right) \left( \int_{M}\left\vert \varphi
\right\vert ^{n}d\sigma +\int_{M}\left\vert \varphi \right\vert
^{n}\left\vert \nabla u\right\vert ^{n}d\mu \right) +c\left( n,\varepsilon
\right) \left\Vert \varphi \nabla u\right\Vert _{L^{n}}^{n} \\
&=&\left( 1+\varepsilon \right) \int_{M}\left\vert \varphi \right\vert
^{n}d\sigma +c_{1}\left( n,\varepsilon \right) \int_{M}\left\vert \varphi
\right\vert ^{n}\left\vert \nabla u\right\vert ^{n}d\mu .
\end{eqnarray*}%
Note that we can find a $\varphi \in C^{\infty }\left( M\right) $ such that $%
\left. \varphi \right\vert _{B_{r}\left( x\right) }=1$ for some $r>0$ and%
\begin{equation*}
\left( 1+\varepsilon \right) \int_{M}\left\vert \varphi \right\vert
^{n}d\sigma +c_{1}\left( n,\varepsilon \right) \int_{M}\left\vert \varphi
\right\vert ^{n}\left\vert \nabla u\right\vert ^{n}d\mu <\frac{1}{p_{1}^{n-1}%
}.
\end{equation*}%
Hence for $i$ large enough, we have%
\begin{equation*}
\left\Vert \nabla \left( \varphi v_{i}\right) \right\Vert _{L^{n}}^{n}<\frac{%
1}{p_{1}^{n-1}}.
\end{equation*}%
In other words,%
\begin{equation*}
\left\Vert \nabla \left( \varphi v_{i}\right) \right\Vert _{L^{n}}^{\frac{n}{%
n-1}}<\frac{1}{p_{1}}.
\end{equation*}%
We have%
\begin{eqnarray*}
\int_{B_{r}\left( x\right) }e^{a_{n}p_{1}\left\vert v_{i}-\overline{\varphi
v_{i}}\right\vert ^{\frac{n}{n-1}}}d\mu &\leq
&\int_{M}e^{a_{n}p_{1}\left\vert \varphi v_{i}-\overline{\varphi v_{i}}%
\right\vert ^{\frac{n}{n-1}}}d\mu \\
&\leq &\int_{M}e^{a_{n}\frac{\left\vert \varphi v_{i}-\overline{\varphi v_{i}%
}\right\vert ^{\frac{n}{n-1}}}{\left\Vert \nabla \left( \varphi v_{i}\right)
\right\Vert _{L^{n}}^{\frac{n}{n-1}}}}d\mu \\
&\leq &c\left( M,g\right) .
\end{eqnarray*}%
Next we observe that%
\begin{eqnarray*}
&&\left\vert u_{i}\right\vert ^{\frac{n}{n-1}} \\
&=&\left\vert \left( v_{i}-\overline{\varphi v_{i}}\right) +u+\overline{%
\varphi v_{i}}\right\vert ^{\frac{n}{n-1}} \\
&\leq &\left( 1+\varepsilon \right) \left\vert v_{i}-\overline{\varphi v_{i}}%
\right\vert ^{\frac{n}{n-1}}+c\left( n,\varepsilon \right) \left\vert
u\right\vert ^{\frac{n}{n-1}}+c\left( n,\varepsilon \right) \left\vert 
\overline{\varphi v_{i}}\right\vert ^{\frac{n}{n-1}},
\end{eqnarray*}%
hence%
\begin{equation*}
e^{a_{n}\left\vert u_{i}\right\vert ^{\frac{n}{n-1}}}\leq e^{\left(
1+\varepsilon \right) a_{n}\left\vert v_{i}-\overline{\varphi v_{i}}%
\right\vert ^{\frac{n}{n-1}}}e^{c\left( n,\varepsilon \right) \left\vert
u\right\vert ^{\frac{n}{n-1}}}e^{c\left( n,\varepsilon \right) \left\vert 
\overline{\varphi v_{i}}\right\vert ^{\frac{n}{n-1}}}.
\end{equation*}%
Since $e^{\left( 1+\varepsilon \right) a_{n}\left\vert v_{i}-\overline{%
\varphi v_{i}}\right\vert ^{\frac{n}{n-1}}}$ is bounded in $L^{\frac{p_{1}}{%
1+\varepsilon }}\left( B_{r}\left( x\right) \right) $, $e^{c\left(
n,\varepsilon \right) \left\vert u\right\vert ^{\frac{n}{n-1}}}\in
L^{q}\left( B_{r}\left( x\right) \right) $ for any $0<q<\infty $ (by Lemma %
\ref{lem2.1}), $e^{c\left( n,\varepsilon \right) \left\vert \overline{%
\varphi v_{i}}\right\vert ^{\frac{n}{n-1}}}\rightarrow 1$ as $i\rightarrow
\infty $ and $\frac{p_{1}}{1+\varepsilon }>p$, it follows from Holder
inequality that $e^{a_{n}\left\vert u_{i}\right\vert ^{\frac{n}{n-1}}}$ is
bounded in $L^{p}\left( B_{r}\left( x\right) \right) $.
\end{proof}

\begin{corollary}
\label{cor2.1}Assume $u_{i}\in W^{1,n}\left( M\right) $ such that $\overline{%
u_{i}}=0$ and $\left\Vert \nabla u_{i}\right\Vert _{L^{n}}\leq 1$. We also
assume $u_{i}\rightharpoonup u$ weakly in $W^{1,n}\left( M\right) $ and%
\begin{equation}
\left\vert \nabla u_{i}\right\vert ^{n}d\mu \rightarrow \left\vert \nabla
u\right\vert ^{n}d\mu +\sigma  \label{eq2.8}
\end{equation}%
as measure. Let $K$ be a compact subset of $M$ and%
\begin{equation}
\kappa =\max_{x\in K}\sigma \left( \left\{ x\right\} \right) .  \label{eq2.9}
\end{equation}

\begin{enumerate}
\item If $\kappa <1$, then for any $1\leq p<\kappa ^{-\frac{1}{n-1}}$, 
\begin{equation}
\sup_{i}\int_{K}e^{a_{n}p\left\vert u_{i}\right\vert ^{\frac{n}{n-1}}}d\mu
<\infty .  \label{eq2.10}
\end{equation}

\item If $\kappa =1$, then $\sigma =\delta _{x_{0}}$ for some $x_{0}\in K$, $%
u=0$ and after passing to a subsequence,%
\begin{equation}
e^{a_{n}\left\vert u_{i}\right\vert ^{\frac{n}{n-1}}}\rightarrow
1+c_{0}\delta _{x_{0}}  \label{eq2.11}
\end{equation}%
as measure for some $c_{0}\geq 0$.
\end{enumerate}
\end{corollary}

\begin{proof}
First assume $\kappa <1$. For any $x\in K$, we have%
\begin{equation}
1\leq p<\kappa ^{-\frac{1}{n-1}}\leq \sigma \left( \left\{ x\right\} \right)
^{-\frac{1}{n-1}}.  \label{eq2.12}
\end{equation}%
By the Proposition \ref{prop2.1} we can find $r_{x}>0$ such that%
\begin{equation}
\sup_{i}\int_{B_{r_{x}}\left( x\right) }e^{a_{n}p\left\vert u_{i}\right\vert
^{\frac{n}{n-1}}}d\mu <\infty .  \label{eq2.13}
\end{equation}%
We have an open covering%
\begin{equation*}
K\subset \dbigcup\limits_{x\in K}B_{r_{x}}\left( x\right) ,
\end{equation*}%
hence there exists $x_{1},\cdots ,x_{N}\in K$ such that%
\begin{equation*}
K\subset \dbigcup\limits_{k=1}^{N}B_{r_{k}}\left( x_{k}\right) .
\end{equation*}%
Here $r_{k}=r_{x_{k}}$. For any $i$,%
\begin{eqnarray*}
\int_{K}e^{a_{n}p\left\vert u_{i}\right\vert ^{\frac{n}{n-1}}}d\mu &\leq
&\sum_{k=1}^{N}\int_{B_{r_{k}\left( x_{k}\right) }}e^{a_{n}p\left\vert
u_{i}\right\vert ^{\frac{n}{n-1}}}d\mu \\
&\leq &\sum_{k=1}^{N}\sup_{j}\int_{B_{r_{k}\left( x_{k}\right)
}}e^{a_{n}p\left\vert u_{j}\right\vert ^{\frac{n}{n-1}}}d\mu \\
&<&\infty .
\end{eqnarray*}

Next assume $\kappa =1$, then for some $x_{0}\in K$, $\sigma \left( \left\{
x_{0}\right\} \right) =1$. Since%
\begin{equation*}
\int_{M}\left\vert \nabla u\right\vert ^{n}d\mu +\sigma \left( M\right) \leq
1,
\end{equation*}%
we see $u$ must be a constant function and $\sigma =\delta _{x_{0}}$. Using $%
\overline{u}=0$, we see $u=0$. After passing to a subsequence, we can assume 
$u_{i}\rightarrow 0$ a.e. For any $r>0$, it follows from the first case that 
$e^{a_{n}\left\vert u_{i}\right\vert ^{\frac{n}{n-1}}}$ is bounded in $%
L^{q}\left( M\backslash B_{r}\left( x_{0}\right) \right) $ for any $q<\infty 
$, hence $e^{a_{n}\left\vert u_{i}\right\vert ^{\frac{n}{n-1}}}\rightarrow 1$
in $L^{1}\left( M\backslash B_{r}\left( x_{0}\right) \right) $. This
together with the fact%
\begin{equation*}
\int_{M}e^{a_{n}\left\vert u_{i}\right\vert ^{\frac{n}{n-1}}}d\mu \leq
c\left( M,g\right)
\end{equation*}%
implies that after passing to a subsequence, $e^{a_{n}\left\vert
u_{i}\right\vert ^{\frac{n}{n-1}}}\rightarrow 1+c_{0}\delta _{x_{0}}$ as
measure for some $c_{0}\geq 0$.
\end{proof}

\begin{theorem}
\label{thm2.1}Assume $\alpha >0$, $m_{i}>0$, $m_{i}\rightarrow \infty $, $%
u_{i}\in W^{1,n}\left( M^{n}\right) $ such that $\overline{u_{i}}=0$ and%
\begin{equation}
\log \int_{M}e^{nm_{i}u_{i}}d\mu \geq \alpha m_{i}^{n}.  \label{eq2.14}
\end{equation}%
We also assume $u_{i}\rightharpoonup u$ weakly in $W^{1,n}\left( M\right) $, 
$\left\vert \nabla u_{i}\right\vert ^{n}d\mu \rightarrow \left\vert \nabla
u\right\vert ^{n}d\mu +\sigma $ as measure and%
\begin{equation}
\frac{e^{nm_{i}u_{i}}}{\int_{M}e^{nm_{i}u_{i}}d\mu }\rightarrow \nu
\label{eq2.15}
\end{equation}%
as measure. Let%
\begin{equation}
\left\{ x\in M:\sigma \left( \left\{ x\right\} \right) \geq \alpha
_{n}^{-1}\alpha \right\} =\left\{ x_{1},\cdots ,x_{N}\right\} ,
\label{eq2.16}
\end{equation}%
here%
\begin{equation}
\alpha _{n}=\left( \frac{n-1}{n}\right) ^{n-1}\frac{1}{\left\vert \mathbb{S}%
^{n-1}\right\vert },  \label{eq2.17}
\end{equation}%
then%
\begin{equation}
\nu =\sum_{i=1}^{N}\nu _{i}\delta _{x_{i}},  \label{eq2.18}
\end{equation}%
here $\nu _{i}\geq 0$ and $\sum_{i=1}^{N}\nu _{i}=1$.
\end{theorem}

\begin{proof}
Assume $x\in M$ such that $\sigma \left( \left\{ x\right\} \right) <\alpha
_{n}^{-1}\alpha $, then we claim that for some $r>0$, $\nu \left(
B_{r}\left( x\right) \right) =0$. Indeed we fix $p$ such that%
\begin{equation}
\alpha _{n}^{\frac{1}{n-1}}\alpha ^{-\frac{1}{n-1}}<p<\sigma \left( \left\{
x\right\} \right) ^{-\frac{1}{n-1}},  \label{eq2.19}
\end{equation}%
it follows from Proposition \ref{prop2.1} that for some $r>0$,%
\begin{equation}
\int_{B_{r}\left( x\right) }e^{a_{n}p\left\vert u_{i}\right\vert ^{\frac{n}{%
n-1}}}d\mu \leq c,  \label{eq2.20}
\end{equation}%
here $c$ is a positive constant independent of $i$. By Young's inequality we
have%
\begin{equation}
nm_{i}u_{i}\leq a_{n}p\left\vert u_{i}\right\vert ^{\frac{n}{n-1}}+\frac{%
\alpha _{n}m_{i}^{n}}{p^{n-1}},  \label{eq2.21}
\end{equation}%
hence%
\begin{equation*}
\int_{B_{r}\left( x\right) }e^{nm_{i}u_{i}}d\mu \leq ce^{\frac{\alpha
_{n}m_{i}^{n}}{p^{n-1}}}.
\end{equation*}%
It follows that%
\begin{equation*}
\frac{\int_{B_{r}\left( x\right) }e^{nm_{i}u_{i}}d\mu }{%
\int_{M}e^{nm_{i}u_{i}}d\mu }\leq ce^{\left( \frac{\alpha _{n}}{p^{n-1}}%
-\alpha \right) m_{i}^{n}}.
\end{equation*}%
In particular,%
\begin{equation*}
\nu \left( B_{r}\left( x\right) \right) \leq \lim \inf_{i\rightarrow \infty }%
\frac{\int_{B_{r}\left( x\right) }e^{nm_{i}u_{i}}d\mu }{%
\int_{M}e^{nm_{i}u_{i}}d\mu }=0.
\end{equation*}%
We get $\nu \left( B_{r}\left( x\right) \right) =0$. The claim is proved.

Clearly the claim implies%
\begin{equation}
\nu \left( M\backslash \left\{ x_{1},\cdots ,x_{N}\right\} \right) =0.
\label{eq2.22}
\end{equation}%
Hence $\nu =\sum_{i=1}^{N}\nu _{i}\delta _{x_{i}},$ with $\nu _{i}\geq 0$
and $\sum_{i=1}^{N}\nu _{i}=1$.
\end{proof}

It is worth pointing out that the arguments for Proposition \ref{prop2.1}
and Theorem \ref{thm2.1} can be easily modified to work for functions
satisfying various boundary conditions or belonging to higher order Sobolev
spaces. We will discuss these examples in Section \ref{sec4}.

\section{A generalization of Aubin inequality to higher order moments case
on $\mathbb{S}^{n}$\label{sec3}}

Here we will extend Aubin's inequality for functions on $\mathbb{S}^{n}$
with zero first order moments of the area element (see \cite[Corollary 2,
p159]{A}) to higher order moments cases. For $n=2$, this is done in \cite{CH}%
.

First we introduce some notations. For a nonnegative integer $k$, we denote%
\begin{eqnarray}
\mathcal{P}_{k} &=&\mathcal{P}_{k}\left( \mathbb{R}^{n+1}\right) =\left\{ 
\text{all polynomials on }\mathbb{R}^{n+1}\text{ with degree at most }%
k\right\} ;  \label{eq3.1} \\
\overset{\circ }{\mathcal{P}}_{k} &=&\overset{\circ }{\mathcal{P}}_{k}\left( 
\mathbb{R}^{n+1}\right) =\left\{ p\in \mathcal{P}_{k}:\int_{\mathbb{S}%
^{n}}pd\mu =0\right\} .  \label{eq3.2}
\end{eqnarray}%
Here $\mu $ is the standard measure on $\mathbb{S}^{n}$.

\begin{definition}
\label{def3.1}For $m\in \mathbb{N}$, let%
\begin{eqnarray*}
&&\mathcal{N}_{m}\left( \mathbb{S}^{n}\right) \\
&=&\left\{ N\in \mathbb{N}:\exists x_{1},\cdots ,x_{N}\in \mathbb{S}^{n}%
\text{ and }\nu _{1},\cdots ,\nu _{N}\in \left[ 0,\infty \right) \text{ s.t. 
}\nu _{1}+\cdots +\nu _{N}=1\right. \\
&&\left. \text{and for any }p\in \overset{\circ }{\mathcal{P}}_{m}\text{, }%
\nu _{1}p\left( x_{1}\right) +\cdots +\nu _{N}p\left( x_{N}\right) =0\text{.}%
\right\} \\
&=&\left\{ N\in \mathbb{N}:\exists x_{1},\cdots ,x_{N}\in \mathbb{S}^{n}%
\text{ and }\nu _{1},\cdots ,\nu _{N}\in \left[ 0,\infty \right) \text{ s.t.
for any }p\in \mathcal{P}_{m}\text{,}\right. \\
&&\left. \nu _{1}p\left( x_{1}\right) +\cdots +\nu _{N}p\left( x_{N}\right) =%
\frac{1}{\left\vert \mathbb{S}^{n}\right\vert }\int_{\mathbb{S}^{n}}pd\mu
.\right\} \text{,}
\end{eqnarray*}%
and%
\begin{equation}
N_{m}\left( \mathbb{S}^{n}\right) =\min \mathcal{N}_{m}\left( \mathbb{S}%
^{n}\right) .  \label{eq3.3}
\end{equation}
\end{definition}

As in \cite{CH}, every choice of $\nu _{1},\cdots ,\nu _{N}$ and $%
x_{1},\cdots ,x_{N}$ in Definition \ref{def3.1} corresponds to an algorithm
for numerical integration of functions on $\mathbb{S}^{n}$ (see \cite{Co,
HSW} for further discussion).

\begin{theorem}
\label{thm3.1}Assume $u\in W^{1,n}\left( \mathbb{S}^{n}\right) $ such that $%
\int_{\mathbb{S}^{n}}ud\mu =0$ (here $\mu $ is the standard measure on $%
\mathbb{S}^{n}$) and for every $p\in \overset{\circ }{\mathcal{P}}_{m}$, $%
\int_{\mathbb{S}^{n}}pe^{nu}d\mu =0$, then for any $\varepsilon >0$, we have%
\begin{equation}
\log \int_{\mathbb{S}^{n}}e^{nu}d\mu \leq \left( \frac{\alpha _{n}}{%
N_{m}\left( \mathbb{S}^{n}\right) }+\varepsilon \right) \left\Vert \nabla
u\right\Vert _{L^{n}}^{n}+c\left( m,n,\varepsilon \right) .  \label{eq3.4}
\end{equation}%
Here $N_{m}\left( \mathbb{S}^{n}\right) $ is defined in Definition \ref%
{def3.1} and%
\begin{equation*}
\alpha _{n}=\left( \frac{n-1}{n}\right) ^{n-1}\frac{1}{\left\vert \mathbb{S}%
^{n-1}\right\vert }.
\end{equation*}
\end{theorem}

\begin{proof}
Let $\alpha =\frac{\alpha _{n}}{N_{m}\left( \mathbb{S}^{n}\right) }%
+\varepsilon $. If the inequality is not true, then there exists $v_{i}\in
W^{1,n}\left( \mathbb{S}^{n}\right) $ such that $\overline{v_{i}}=0$, $\int_{%
\mathbb{S}^{n}}pe^{nv_{i}}d\mu =0$ for all $p\in \overset{\circ }{\mathcal{P}%
}_{m}$ and%
\begin{equation*}
\log \int_{\mathbb{S}^{n}}e^{nv_{i}}d\mu -\alpha \left\Vert \nabla
v_{i}\right\Vert _{L^{n}}^{n}\rightarrow \infty
\end{equation*}%
as $i\rightarrow \infty $. In particular $\int_{\mathbb{S}%
^{n}}e^{nv_{i}}d\mu \rightarrow \infty $. Since%
\begin{equation*}
\log \int_{\mathbb{S}^{n}}e^{nv_{i}}d\mu \leq \alpha _{n}\left\Vert \nabla
v_{i}\right\Vert _{L^{n}}^{n}+c\left( n\right) ,
\end{equation*}%
we see $\left\Vert \nabla v_{i}\right\Vert _{L^{n}}\rightarrow \infty $. Let 
$m_{i}=\left\Vert \nabla v_{i}\right\Vert _{L^{n}}$ and $u_{i}=\frac{v_{i}}{%
m_{i}}$, then $m_{i}\rightarrow \infty $, $\left\Vert \nabla
u_{i}\right\Vert _{L^{n}}=1$ and $\overline{u_{i}}=0$. After passing to a
subsequence, we have%
\begin{eqnarray*}
u_{i} &\rightharpoonup &u\text{ weakly in }W^{1,n}\left( \mathbb{S}%
^{n}\right) ; \\
\log \int_{\mathbb{S}^{n}}e^{nm_{i}u_{i}}d\mu -\alpha m_{i}^{n} &\rightarrow
&\infty ; \\
\left\vert \nabla u_{i}\right\vert ^{n}d\mu &\rightarrow &\left\vert \nabla
u\right\vert ^{n}d\mu +\sigma \text{ as measure;} \\
\frac{e^{nm_{i}u_{i}}}{\int_{\mathbb{S}^{n}}e^{nm_{i}u_{i}}d\mu }
&\rightarrow &\nu \text{ as measure.}
\end{eqnarray*}%
Let%
\begin{equation}
\left\{ x\in \mathbb{S}^{n}:\sigma \left( \left\{ x\right\} \right) \geq
\alpha _{n}^{-1}\alpha \right\} =\left\{ x_{1},\cdots ,x_{N}\right\} ,
\label{eq3.5}
\end{equation}%
then it follows from Theorem \ref{thm2.1} that%
\begin{equation}
\nu =\sum_{i=1}^{N}\nu _{i}\delta _{x_{i}}.  \label{eq3.6}
\end{equation}%
Moreover $\nu _{i}\geq 0$, $\sum_{i=1}^{N}\nu _{i}=1$ and%
\begin{equation*}
\sum_{i=1}^{N}\nu _{i}p\left( x_{i}\right) =0
\end{equation*}%
for any $p\in \overset{\circ }{\mathcal{P}}_{m}$. Hence $N\in \mathcal{N}%
_{m}\left( \mathbb{S}^{n}\right) $. It follows from $\sigma \left( \mathbb{S}%
^{n}\right) \leq 1$ and (\ref{eq3.5}) that%
\begin{equation*}
\alpha _{n}^{-1}\alpha N\leq 1,
\end{equation*}%
We get%
\begin{equation*}
\alpha \leq \frac{\alpha _{n}}{N}\leq \frac{\alpha _{n}}{N_{m}\left( \mathbb{%
S}^{n}\right) }.
\end{equation*}%
This contradicts with the choice of $\alpha $.
\end{proof}

Indeed, what we get from the above argument is the following

\begin{theorem}
\label{thm3.2}If $u\in W^{1,n}\left( \mathbb{S}^{n}\right) $ such that $%
\int_{\mathbb{S}^{n}}ud\mu =0$ (here $\mu $ is the standard measure on $%
\mathbb{S}^{n}$) and for every $p\in \overset{\circ }{\mathcal{P}}_{m}$,%
\begin{equation}
\left\vert \int_{\mathbb{S}^{n}}pe^{nu}d\mu \right\vert \leq b\left(
p\right) ,  \label{eq3.7}
\end{equation}%
here $b\left( p\right) $ is a nonnegative number depending only on $p$, then
for any $\varepsilon >0$,%
\begin{equation}
\log \int_{\mathbb{S}^{n}}e^{nu}d\mu \leq \left( \frac{\alpha _{n}}{%
N_{m}\left( \mathbb{S}^{n}\right) }+\varepsilon \right) \left\Vert \nabla
u\right\Vert _{L^{n}}^{n}+c\left( m,n,b,\varepsilon \right) .  \label{eq3.8}
\end{equation}%
Here $N_{m}\left( \mathbb{S}^{n}\right) $ is defined in Definition \ref%
{def3.1} and%
\begin{equation*}
\alpha _{n}=\left( \frac{n-1}{n}\right) ^{n-1}\frac{1}{\left\vert \mathbb{S}%
^{n-1}\right\vert }.
\end{equation*}
\end{theorem}

We remark that the constant $\frac{\alpha _{n}}{N_{m}\left( \mathbb{S}%
^{n}\right) }+\varepsilon $ is almost optimal in the following sense: If $%
a\geq 0$ and $c\in \mathbb{R}$ such that for any $u\in W^{1,n}\left( \mathbb{%
S}^{n}\right) $ with $\overline{u}=0$ and $\int_{\mathbb{S}^{n}}pe^{nu}d\mu
=0$ for every $p\in \overset{\circ }{\mathcal{P}}_{m}$, we have%
\begin{equation}
\log \int_{\mathbb{S}^{n}}e^{nu}d\mu \leq a\left\Vert \nabla u\right\Vert
_{L^{n}}^{n}+c,  \label{eq3.9}
\end{equation}%
then $a\geq \frac{\alpha _{n}}{N_{m}\left( \mathbb{S}^{n}\right) }$. This
claim can be proved almost the same way as the argument in \cite[Lemma 3.1]%
{CH}. For reader's convenience we sketch the proof here.

First we note that we can rewrite the assumption as for any $u\in
W^{1,n}\left( \mathbb{S}^{n}\right) $ with $\int_{\mathbb{S}^{n}}pe^{nu}d\mu
=0$ for every $p\in \overset{\circ }{\mathcal{P}}_{m}$, we have%
\begin{equation}
\log \int_{\mathbb{S}^{n}}e^{nu}d\mu \leq a\left\Vert \nabla u\right\Vert
_{L^{n}}^{n}+n\overline{u}+c.  \label{eq3.10}
\end{equation}

Assume $N\in \mathbb{N}$, $x_{1},\cdots ,x_{N}\in \mathbb{S}^{n}$ and $\nu
_{1},\cdots ,\nu _{N}\in \left[ 0,\infty \right) $ s.t. $\nu _{1}+\cdots
+\nu _{N}=1$ and for any $p\in \overset{\circ }{\mathcal{P}}_{m}$, $\nu
_{1}p\left( x_{1}\right) +\cdots +\nu _{N}p\left( x_{N}\right) =0$. We will
prove $a\geq \frac{\alpha _{n}}{N}$. The above remark follows. Without
losing of generality we can assume $\nu _{i}>0$ for $1\leq i\leq N$ and $%
x_{i}\neq x_{j}$ for $1\leq i<j\leq N$.

For $x,y\in \mathbb{S}^{n}$, we denote $\overline{xy}$ as the geodesic
distance between $x$ and $y$ on $\mathbb{S}^{n}$. For $r>0$ and $x\in 
\mathbb{S}^{n}$, we denote $B_{r}\left( x\right) $ as the geodesic ball with
radius $r$ and center $x$ i.e. $B_{r}\left( x\right) =\left\{ y\in \mathbb{S}%
^{n}:\overline{xy}<r\right\} $.

Let $\delta >0$ be small enough such that for $1\leq i<j\leq N$, $\overline{%
B_{2\delta }\left( x_{i}\right) }\cap \overline{B_{2\delta }\left(
x_{j}\right) }=\emptyset $. For $0<\varepsilon <\delta $, we let%
\begin{equation}
\phi _{\varepsilon }\left( t\right) =\left\{ 
\begin{array}{cc}
\frac{n}{n-1}\log \frac{\delta }{\varepsilon }, & 0<t<\varepsilon ; \\ 
\frac{n}{n-1}\log \frac{\delta }{t}, & \varepsilon <t<\delta ; \\ 
0, & t>\delta .%
\end{array}%
\right.  \label{eq3.11}
\end{equation}%
If $b\in \mathbb{R}$, then we write%
\begin{equation}
\phi _{\varepsilon ,b}\left( t\right) =\left\{ 
\begin{array}{cc}
\phi _{\varepsilon }\left( t\right) +b, & 0<t<\delta ; \\ 
b\left( 2-\frac{t}{\delta }\right) , & \delta <t<2\delta ; \\ 
0, & t>2\delta .%
\end{array}%
\right.  \label{eq3.12}
\end{equation}

Let%
\begin{equation}
v\left( x\right) =\sum_{i=1}^{N}\phi _{\varepsilon ,\frac{1}{n}\log \nu
_{i}}\left( \overline{xx_{i}}\right) ,  \label{eq3.13}
\end{equation}%
then calculation shows%
\begin{equation}
\int_{\mathbb{S}^{n}}e^{nv}d\mu =\left\vert \mathbb{S}^{n-1}\right\vert
\delta ^{\frac{n^{2}}{n-1}}\varepsilon ^{-\frac{n}{n-1}}+O\left( \log \frac{1%
}{\varepsilon }\right)  \label{eq3.14}
\end{equation}%
as $\varepsilon \rightarrow 0^{+}$.

For $p\in \overset{\circ }{\mathcal{P}}_{m}$, using%
\begin{equation}
\dsum\limits_{i=1}^{N}\nu _{i}p\left( x_{i}\right) =0,  \label{eq3.15}
\end{equation}%
we can show%
\begin{equation}
\int_{\mathbb{S}^{n}}e^{nv}pd\mu =O\left( \log \frac{1}{\varepsilon }\right)
\label{eq3.16}
\end{equation}%
as $\varepsilon \rightarrow 0^{+}$.

To get a test function satisfying orthogornality condition, we need to do
some corrections. Let us fix a base of $\left. \overset{\circ }{\mathcal{P}}%
_{m}\right\vert _{\mathbb{S}^{n}}$, namely $\left. p_{1}\right\vert _{%
\mathbb{S}^{n}},\cdots ,\left. p_{l}\right\vert _{\mathbb{S}^{n}}$, here $%
p_{1},\cdots ,p_{l}\in \overset{\circ }{\mathcal{P}}_{m}$. We first claim
that there exists $\psi _{1},\cdots ,\psi _{l}\in C_{c}^{\infty }\left( 
\mathbb{S}^{n}\backslash \dbigcup\limits_{i=1}^{N}\overline{B_{2\delta
}\left( x_{i}\right) }\right) $ such that the determinant%
\begin{equation}
\det \left[ \int_{\mathbb{S}^{n}}\psi _{j}p_{k}d\mu \right] _{1\leq j,k\leq
l}\neq 0.  \label{eq3.17}
\end{equation}%
Indeed, fix a nonzero smooth function $\eta \in C_{c}^{\infty }\left( 
\mathbb{S}^{n}\backslash \dbigcup\limits_{i=1}^{N}\overline{B_{2\delta
}\left( x_{i}\right) }\right) $, then $\eta p_{1},\cdots ,\eta p_{l}$ are
linearly independent. It follows that the matrix%
\begin{equation*}
\left[ \int_{\mathbb{S}^{n}}\eta ^{2}p_{j}p_{k}d\mu \right] _{1\leq j,k\leq
l}
\end{equation*}%
is positive definite. Then $\psi _{j}=\eta ^{2}p_{j}$ satisfies the claim.

It follows from (\ref{eq3.17}) that we can find $\beta _{1},\cdots ,\beta
_{l}\in \mathbb{R}$ such that%
\begin{equation}
\int_{\mathbb{S}^{n}}\left( e^{nv}+\sum_{j=1}^{l}\beta _{j}\psi _{j}\right)
p_{k}d\mu =0  \label{eq3.18}
\end{equation}%
for $k=1,\cdots ,l$. Moreover%
\begin{equation}
\beta _{j}=O\left( \log \frac{1}{\varepsilon }\right)  \label{eq3.19}
\end{equation}%
as $\varepsilon \rightarrow 0^{+}$. As a consequence we can find a constant $%
c_{1}>0$ such that%
\begin{equation}
\sum_{j=1}^{l}\beta _{j}\psi _{j}+c_{1}\log \frac{1}{\varepsilon }\geq \log 
\frac{1}{\varepsilon }.  \label{eq3.20}
\end{equation}%
We define $u$ as%
\begin{equation}
e^{nu}=e^{nv}+\sum_{j=1}^{l}\beta _{j}\psi _{j}+c_{1}\log \frac{1}{%
\varepsilon }.  \label{eq3.21}
\end{equation}%
Note this $u$ will be the test function we use to prove our remark.

It follows from (\ref{eq3.18}) that $\int_{\mathbb{S}^{n}}e^{nu}pd\mu =0$
for all $p\in \overset{\circ }{\mathcal{P}}_{m}$. Moreover using (\ref%
{eq3.14}) and (\ref{eq3.19}) we see%
\begin{eqnarray}
\int_{\mathbb{S}^{n}}e^{nu}d\mu &=&\left\vert \mathbb{S}^{n-1}\right\vert
\delta ^{\frac{n^{2}}{n-1}}\varepsilon ^{-\frac{n}{n-1}}+O\left( \log \frac{1%
}{\varepsilon }\right)  \label{eq3.22} \\
&=&\left\vert \mathbb{S}^{n-1}\right\vert \delta ^{\frac{n^{2}}{n-1}%
}\varepsilon ^{-\frac{n}{n-1}}\left( 1+o\left( 1\right) \right) ,  \notag
\end{eqnarray}%
hence%
\begin{equation}
\log \int_{\mathbb{S}^{n}}e^{nu}d\mu =\frac{n}{n-1}\log \frac{1}{\varepsilon 
}+O\left( 1\right)  \label{eq3.23}
\end{equation}%
as $\varepsilon \rightarrow 0^{+}$. Calculation shows%
\begin{equation}
\overline{u}=o\left( \log \frac{1}{\varepsilon }\right)  \label{eq3.24}
\end{equation}%
and%
\begin{equation}
\int_{\mathbb{S}^{n}}\left\vert \nabla u\right\vert ^{n}d\mu =\left( \frac{n%
}{n-1}\right) ^{n}\left\vert \mathbb{S}^{n-1}\right\vert N\log \frac{1}{%
\varepsilon }+o\left( \log \frac{1}{\varepsilon }\right) .  \label{eq3.25}
\end{equation}%
We plug $u$ into (\ref{eq3.10}) and get%
\begin{equation*}
\frac{n}{n-1}\log \frac{1}{\varepsilon }\leq \left( \frac{n}{n-1}\right)
^{n}\left\vert \mathbb{S}^{n-1}\right\vert Na\log \frac{1}{\varepsilon }%
+o\left( \log \frac{1}{\varepsilon }\right) .
\end{equation*}%
Divide $\log \frac{1}{\varepsilon }$ on both sides and let $\varepsilon
\rightarrow 0^{+}$, we see $a\geq \frac{\alpha _{n}}{N}$.

It is clear that $N_{1}\left( \mathbb{S}^{n}\right) =2$. Hence Aubin's
theorem \cite[Corollary 2 on p159]{A} follows from Theorem \ref{thm3.1}.

\begin{lemma}
\label{lem3.1}$N_{2}\left( \mathbb{S}^{n}\right) =n+2$.
\end{lemma}

\begin{corollary}
\label{cor3.1}Assume $u\in W^{1,n}\left( \mathbb{S}^{n}\right) $ such that $%
\int_{\mathbb{S}^{n}}ud\mu =0$ (here $\mu $ is the standard measure on $%
\mathbb{S}^{n}$) and for every $p\in \overset{\circ }{\mathcal{P}}_{2}$,%
\begin{equation*}
\left\vert \int_{\mathbb{S}^{n}}pe^{nu}d\mu \right\vert \leq b\left(
p\right) ,
\end{equation*}%
here $b\left( p\right) $ is a nonnegative number depending only on $p$, then
for any $\varepsilon >0$, we have%
\begin{equation}
\log \int_{\mathbb{S}^{n}}e^{nu}d\mu \leq \left( \frac{\alpha _{n}}{n+2}%
+\varepsilon \right) \left\Vert \nabla u\right\Vert _{L^{n}}^{n}+c\left(
n,b,\varepsilon \right) .  \label{eq3.26}
\end{equation}
\end{corollary}

We will prove Lemma \ref{lem3.1} after some preparations. In $\mathbb{R}^{N}$%
, we have the hyperplane%
\begin{equation}
H=\left\{ x\in \mathbb{R}^{N}:x_{1}+\cdots +x_{N}=1\right\} .  \label{eq3.27}
\end{equation}%
Here $x_{1},\cdots ,x_{N}$ are the coordinates of $x$. Let $e_{1},\cdots
,e_{N}$ be the standard base of $\mathbb{R}^{N}$ and%
\begin{equation}
y=\frac{1}{N}\left( e_{1}+\cdots +e_{N}\right) .  \label{eq3.28}
\end{equation}%
We denote%
\begin{equation}
\Sigma =\left\{ x\in H:\left\vert x-y\right\vert =\sqrt{\frac{N-1}{N}}%
\right\} .  \label{eq3.29}
\end{equation}%
Note that $\Sigma $ is $N-2$ dimensional sphere with radius $\sqrt{\frac{N-1%
}{N}}$.

\begin{lemma}
\label{lem3.2}For any $p\in \mathcal{P}_{2}\left( \mathbb{R}^{N}\right) $,
we have%
\begin{equation}
\frac{1}{\left\vert \Sigma \right\vert }\int_{\Sigma }pdS=\frac{1}{N}%
\sum_{i=1}^{N}p\left( e_{i}\right) .  \label{eq3.30}
\end{equation}%
Here $dS$ is the standard measure on $\Sigma $, and $\left\vert \Sigma
\right\vert $ is the measure of $\Sigma $ under $dS$.
\end{lemma}

\begin{proof}
For $1\leq i\leq N$, we have%
\begin{equation*}
\frac{1}{\left\vert \Sigma \right\vert }\int_{\Sigma }x_{i}dS=\frac{1}{%
\left\vert \Sigma \right\vert }\int_{\Sigma }x_{1}dS.
\end{equation*}%
Moreover%
\begin{equation*}
\frac{1}{\left\vert \Sigma \right\vert }\int_{\Sigma }\left( x_{1}+\cdots
+x_{N}\right) dS=1,
\end{equation*}%
hence $\frac{1}{\left\vert \Sigma \right\vert }\int_{\Sigma }x_{1}dS=\frac{1%
}{N}$. It follows that for $1\leq i\leq N$,%
\begin{equation}
\frac{1}{\left\vert \Sigma \right\vert }\int_{\Sigma }x_{i}dS=\frac{1}{N}.
\label{eq3.31}
\end{equation}

To continue we observe that for $1\leq i\leq N$,%
\begin{equation*}
\frac{1}{\left\vert \Sigma \right\vert }\int_{\Sigma }x_{i}^{2}dS=\frac{1}{%
\left\vert \Sigma \right\vert }\int_{\Sigma }x_{1}^{2}dS.
\end{equation*}%
On the other hand,%
\begin{equation*}
\frac{1}{\left\vert \Sigma \right\vert }\int_{\Sigma }\sum_{i=1}^{N}\left(
x_{i}-\frac{1}{N}\right) ^{2}dS=\frac{N-1}{N},
\end{equation*}%
developing the identity out we get for $1\leq i\leq N$,%
\begin{equation}
\frac{1}{\left\vert \Sigma \right\vert }\int_{\Sigma }x_{i}^{2}dS=\frac{1}{N}%
.  \label{eq3.32}
\end{equation}%
At last we claim for $1\leq i<j\leq N$,%
\begin{equation}
\frac{1}{\left\vert \Sigma \right\vert }\int_{\Sigma }x_{i}x_{j}dS=0.
\label{eq3.33}
\end{equation}%
Indeed this follows from%
\begin{eqnarray*}
\frac{1}{\left\vert \Sigma \right\vert }\int_{\Sigma }x_{i}x_{j}dS &=&\frac{1%
}{\left\vert \Sigma \right\vert }\int_{\Sigma }x_{1}x_{2}dS; \\
\frac{1}{\left\vert \Sigma \right\vert }\int_{\Sigma }\left( x_{1}+\cdots
+x_{N}\right) ^{2}dS &=&1,
\end{eqnarray*}%
and (\ref{eq3.32}). Lemma \ref{lem3.2} follows from (\ref{eq3.31}), (\ref%
{eq3.32}) and (\ref{eq3.33}).
\end{proof}

Now we are ready to prove Lemma \ref{lem3.1}.

\begin{proof}[Proof of Lemma \protect\ref{lem3.1}]
Assume $N\in \mathcal{N}_{2}\left( \mathbb{S}^{n}\right) $, we claim $N\geq
n+2$. If this is not the case, then $N<n+2$. We can find $x_{1},\cdots
,x_{N}\in \mathbb{S}^{n}$, $\nu _{1},\cdots ,\nu _{N}\geq 0$ such that%
\begin{equation*}
\nu _{1}p\left( x_{1}\right) +\cdots +\nu _{N}p\left( x_{N}\right) =\frac{1}{%
\left\vert \mathbb{S}^{n}\right\vert }\int_{\mathbb{S}^{n}}pd\mu
\end{equation*}%
for every $p\in \mathcal{P}_{2}\left( \mathbb{R}^{n+1}\right) $. In
particular%
\begin{eqnarray*}
\nu _{1}+\cdots +\nu _{N} &=&1; \\
\nu _{1}x_{1}+\cdots +\nu _{N}x_{N} &=&0.
\end{eqnarray*}%
Let%
\begin{equation*}
V=\limfunc{span}\left\{ x_{1},\cdots ,x_{N}\right\} ,
\end{equation*}%
then $\dim V\leq N-1\leq n$. Hence we can find a nonzero vector $\xi \in
V^{\perp }$. Let $p\left( x\right) =\left( \xi \cdot x\right) ^{2}$, then%
\begin{equation*}
0=\nu _{1}p\left( x_{1}\right) +\cdots +\nu _{N}p\left( x_{N}\right) =\frac{1%
}{\left\vert \mathbb{S}^{n}\right\vert }\int_{\mathbb{S}^{n}}pd\mu >0.
\end{equation*}%
A contradiction. Hence $N_{2}\left( \mathbb{S}^{n}\right) \geq n+2$.

On the other hand, it follows from Lemma \ref{lem3.2} and dilation,
translation and orthogornal transformation that the vertex of a regular $%
\left( n+1\right) $-simplex embedded in the unit ball, namely $x_{1},\cdots
,x_{n+2}$, satisfies%
\begin{equation*}
\frac{1}{n+2}p\left( x_{1}\right) +\cdots +\frac{1}{n+2}p\left(
x_{n+2}\right) =\frac{1}{\left\vert \mathbb{S}^{n}\right\vert }\int_{\mathbb{%
S}^{n}}pd\mu
\end{equation*}%
for every $p\in \mathcal{P}_{2}\left( \mathbb{R}^{n+1}\right) $. Hence $%
N_{2}\left( \mathbb{S}^{n}\right) \leq n+2$.
\end{proof}

\section{Further discussions\label{sec4}}

In this section we show that our approach above can be modified without much
effort for higher order Sobolev spaces and Sobolev spaces on surfaces with
nonempty boundary. For reader's convenience we carefully write down theorems
in every case considered, although they seem a little repetitive. Another
reason is these statements are of interest themselves.

\subsection{$W^{s,\frac{n}{s}}\left( M^{n}\right) $ for even $s$\label%
{sec4.1}}

Recall $\left( M^{n},g\right) $ is a smooth compact Riemannian manifold with
dimension $n$. Let $s\in \mathbb{N}$ be an even number strictly less than $n$%
, we have the usual Sobolev space $W^{s,\frac{n}{s}}\left( M\right) $ with
norm%
\begin{equation}
\left\Vert u\right\Vert _{W^{s,\frac{n}{s}}\left( M\right) }=\left(
\sum_{k=0}^{s}\int_{M}\left\vert D^{k}u\right\vert ^{\frac{n}{s}}d\mu
\right) ^{\frac{s}{n}}  \label{eq4.1}
\end{equation}%
for $u\in W^{s,\frac{n}{s}}\left( M\right) $. Here $D^{k}$ is the
differentiation associated with the natural connection of $g$. Standard
elliptic theory tells us%
\begin{equation}
\left\Vert u\right\Vert _{W^{s,\frac{n}{s}}}\leq c\left( M,g\right) \left(
\left\Vert \Delta ^{\frac{s}{2}}u\right\Vert _{L^{\frac{n}{s}}}+\left\Vert
u\right\Vert _{L^{\frac{n}{s}}}\right) .  \label{eq4.2}
\end{equation}%
We also have the Poincare inequality%
\begin{equation}
\left\Vert u-\overline{u}\right\Vert _{L^{\frac{n}{s}}}\leq c\left(
M,g\right) \left\Vert \Delta ^{\frac{s}{2}}u\right\Vert _{L^{\frac{n}{s}}}
\label{eq4.3}
\end{equation}%
for $u\in W^{s,\frac{n}{s}}\left( M\right) $.

Let%
\begin{equation}
a_{s,n}=\frac{n}{\left\vert \mathbb{S}^{n-1}\right\vert }\left( \frac{\pi ^{%
\frac{n}{2}}2^{s}\Gamma \left( \frac{s}{2}\right) }{\Gamma \left( \frac{n-s}{%
2}\right) }\right) ^{\frac{n}{n-s}}.  \label{eq4.4}
\end{equation}%
Here $\Gamma $ is the usual Gamma function i.e.%
\begin{equation}
\Gamma \left( \alpha \right) =\int_{0}^{\infty }t^{\alpha -1}e^{-t}dt
\label{eq4.5}
\end{equation}%
for $\alpha >0$. The Moser-Trudinger inequality (see \cite{BCY, F}) tells us
that for every $u\in W^{s,\frac{n}{s}}\left( M\right) \backslash \left\{
0\right\} $ with $\overline{u}=0$, we have%
\begin{equation}
\int_{M}\exp \left( a_{s,n}\frac{\left\vert u\right\vert ^{\frac{n}{n-s}}}{%
\left\Vert \Delta ^{\frac{s}{2}}u\right\Vert _{L^{\frac{n}{s}}}^{\frac{n}{n-s%
}}}\right) d\mu \leq c\left( M,g\right) .  \label{eq4.6}
\end{equation}%
It follows from (\ref{eq4.6}) and Young's inequality that for any $u\in W^{s,%
\frac{n}{s}}\left( M\right) $ with $\overline{u}=0$, we have the
Moser-Trudinger-Onofri inequality%
\begin{equation}
\log \int_{M}e^{nu}d\mu \leq \alpha _{s,n}\left\Vert \Delta ^{\frac{s}{2}%
}u\right\Vert _{L^{\frac{n}{s}}}^{\frac{n}{s}}+c_{1}\left( M,g\right) .
\label{eq4.7}
\end{equation}%
Here%
\begin{equation}
\alpha _{s,n}=s\left( \frac{n-s}{n}\left\vert \mathbb{S}^{n-1}\right\vert
\right) ^{\frac{n-s}{s}}\left( \frac{\Gamma \left( \frac{n-s}{2}\right) }{%
\pi ^{\frac{n}{2}}2^{s}\Gamma \left( \frac{s}{2}\right) }\right) ^{\frac{n}{s%
}}.  \label{eq4.8}
\end{equation}

Again we start from a basic qualitative property of functions in $W^{s,\frac{%
n}{s}}\left( M\right) $.

\begin{lemma}
\label{lem4.1}For any $u\in W^{s,\frac{n}{s}}\left( M\right) $ and $a>0$, we
have%
\begin{equation}
\int_{M}e^{a\left\vert u\right\vert ^{\frac{n}{n-s}}}d\mu <\infty .
\label{eq4.9}
\end{equation}
\end{lemma}

This follows from (\ref{eq4.6}) exactly in the same way as the proof of
Lemma \ref{lem2.1}. Indeed the argument there is written in a way working
for higher order Sobolev spaces as well.

Next we will derive an analogy of Proposition \ref{prop2.1} for $W^{s,\frac{n%
}{s}}\left( M\right) $.

\begin{proposition}
\label{prop4.1}Assume $s\in \mathbb{N}$ is an even number strictly less than 
$n$, $u_{i}\in W^{s,\frac{n}{s}}\left( M^{n}\right) $ such that $\overline{%
u_{i}}=0$, $u_{i}\rightharpoonup u$ weakly in $W^{s,\frac{n}{s}}\left(
M\right) $ and%
\begin{equation}
\left\vert \Delta ^{\frac{s}{2}}u_{i}\right\vert ^{\frac{n}{s}}d\mu
\rightarrow \left\vert \Delta ^{\frac{s}{2}}u\right\vert ^{\frac{n}{s}}d\mu
+\sigma  \label{eq4.10}
\end{equation}%
as measure. If $x\in M$ and $p\in \mathbb{R}$ such that $0<p<\sigma \left(
\left\{ x\right\} \right) ^{-\frac{s}{n-s}}$, then for some $r>0$,%
\begin{equation}
\sup_{i}\int_{B_{r}\left( x\right) }e^{a_{s,n}p\left\vert u_{i}\right\vert ^{%
\frac{n}{n-s}}}d\mu <\infty .  \label{eq4.11}
\end{equation}%
Here%
\begin{equation}
a_{s,n}=\frac{n}{\left\vert \mathbb{S}^{n-1}\right\vert }\left( \frac{\pi ^{%
\frac{n}{2}}2^{s}\Gamma \left( \frac{s}{2}\right) }{\Gamma \left( \frac{n-s}{%
2}\right) }\right) ^{\frac{n}{n-s}}.  \label{eq4.12}
\end{equation}
\end{proposition}

\begin{proof}
Fix $p_{1}\in \left( p,\sigma \left( \left\{ x\right\} \right) ^{-\frac{s}{%
n-s}}\right) $, then%
\begin{equation}
\sigma \left( \left\{ x\right\} \right) <\frac{1}{p_{1}^{\frac{n-s}{s}}}.
\label{eq4.13}
\end{equation}%
We can find a $\varepsilon >0$ such that%
\begin{equation}
\left( 1+\varepsilon \right) \sigma \left( \left\{ x\right\} \right) <\frac{1%
}{p_{1}^{\frac{n-s}{s}}}  \label{eq4.14}
\end{equation}%
and%
\begin{equation}
\left( 1+\varepsilon \right) p<p_{1}.  \label{eq4.15}
\end{equation}%
Let $v_{i}=u_{i}-u$, then $v_{i}\rightharpoonup 0$ weakly in $W^{s,\frac{n}{s%
}}\left( M\right) $, $v_{i}\rightarrow 0$ in $W^{s-1,\frac{n}{s}}\left(
M\right) $. For any $\varphi \in C^{\infty }\left( M\right) $, we have%
\begin{eqnarray*}
\left\vert \Delta ^{\frac{s}{2}}\left( \varphi v_{i}\right) \right\vert
&\leq &\left\vert \varphi \Delta ^{\frac{s}{2}}v_{i}\right\vert
+c\sum_{k=0}^{s-1}\left\vert D^{s-k}\varphi \right\vert \left\vert
D^{k}v_{i}\right\vert \\
&\leq &\left\vert \varphi \Delta ^{\frac{s}{2}}u_{i}\right\vert +\left\vert
\varphi \Delta ^{\frac{s}{2}}u\right\vert +c\sum_{k=0}^{s-1}\left\vert
D^{s-k}\varphi \right\vert \left\vert D^{k}v_{i}\right\vert .
\end{eqnarray*}%
Hence%
\begin{eqnarray*}
&&\left\Vert \Delta ^{\frac{s}{2}}\left( \varphi v_{i}\right) \right\Vert
_{L^{\frac{n}{s}}}^{\frac{n}{s}} \\
&\leq &\left( \left\Vert \varphi \Delta ^{\frac{s}{2}}u_{i}\right\Vert _{L^{%
\frac{n}{s}}}+\left\Vert \varphi \Delta ^{\frac{s}{2}}u\right\Vert _{L^{%
\frac{n}{s}}}+c\sum_{k=0}^{s-1}\left\Vert \left\vert D^{s-k}\varphi
\right\vert \left\vert D^{k}v_{i}\right\vert \right\Vert _{L^{\frac{n}{s}%
}}\right) ^{\frac{n}{s}} \\
&\leq &\left( 1+\varepsilon \right) \left\Vert \varphi \Delta ^{\frac{s}{2}%
}u_{i}\right\Vert _{L^{\frac{n}{s}}}^{\frac{n}{s}}+c\left( \varepsilon
\right) \left\Vert \varphi \Delta ^{\frac{s}{2}}u\right\Vert _{L^{\frac{n}{s}%
}}^{\frac{n}{s}}+c\left( \varepsilon \right) \sum_{k=0}^{s-1}\left\Vert
\left\vert D^{s-k}\varphi \right\vert \left\vert D^{k}v_{i}\right\vert
\right\Vert _{L^{\frac{n}{s}}}^{\frac{n}{s}}.
\end{eqnarray*}%
Letting $i\rightarrow \infty $,%
\begin{eqnarray*}
&&\lim \sup_{i\rightarrow \infty }\left\Vert \Delta ^{\frac{s}{2}}\left(
\varphi v_{i}\right) \right\Vert _{L^{\frac{n}{s}}}^{\frac{n}{s}} \\
&\leq &\left( 1+\varepsilon \right) \left( \int_{M}\left\vert \varphi
\right\vert ^{\frac{n}{s}}d\sigma +\int_{M}\left\vert \varphi \right\vert ^{%
\frac{n}{s}}\left\vert \Delta ^{\frac{s}{2}}u\right\vert ^{\frac{n}{s}}d\mu
\right) +c\left( \varepsilon \right) \left\Vert \varphi \Delta ^{\frac{s}{2}%
}u\right\Vert _{L^{\frac{n}{s}}}^{\frac{n}{s}} \\
&=&\left( 1+\varepsilon \right) \int_{M}\left\vert \varphi \right\vert ^{%
\frac{n}{s}}d\sigma +c_{1}\left( \varepsilon \right) \int_{M}\left\vert
\varphi \right\vert ^{\frac{n}{s}}\left\vert \Delta ^{\frac{s}{2}%
}u\right\vert ^{\frac{n}{s}}d\mu .
\end{eqnarray*}%
We can find a $\varphi \in C^{\infty }\left( M\right) $ such that $\left.
\varphi \right\vert _{B_{r}\left( x\right) }=1$ for some $r>0$ and%
\begin{equation*}
\left( 1+\varepsilon \right) \int_{M}\left\vert \varphi \right\vert ^{\frac{n%
}{s}}d\sigma +c_{1}\left( \varepsilon \right) \int_{M}\left\vert \varphi
\right\vert ^{\frac{n}{s}}\left\vert \Delta ^{\frac{s}{2}}u\right\vert ^{%
\frac{n}{s}}d\mu <\frac{1}{p_{1}^{\frac{n-s}{s}}}.
\end{equation*}%
Hence for $i$ large enough, we have%
\begin{equation*}
\left\Vert \Delta ^{\frac{s}{2}}\left( \varphi v_{i}\right) \right\Vert _{L^{%
\frac{n}{s}}}^{\frac{n}{s}}<\frac{1}{p_{1}^{\frac{n-s}{s}}}.
\end{equation*}%
In other words,%
\begin{equation*}
\left\Vert \Delta ^{\frac{s}{2}}\left( \varphi v_{i}\right) \right\Vert _{L^{%
\frac{n}{s}}}^{\frac{n}{n-s}}<\frac{1}{p_{1}}.
\end{equation*}%
We have%
\begin{eqnarray*}
\int_{B_{r}\left( x\right) }e^{a_{s,n}p_{1}\left\vert v_{i}-\overline{%
\varphi v_{i}}\right\vert ^{\frac{n}{n-s}}}d\mu &\leq
&\int_{M}e^{a_{s,n}p_{1}\left\vert \varphi v_{i}-\overline{\varphi v_{i}}%
\right\vert ^{\frac{n}{n-s}}}d\mu \\
&\leq &\int_{M}e^{a_{s,n}\frac{\left\vert \varphi v_{i}-\overline{\varphi
v_{i}}\right\vert ^{\frac{n}{n-s}}}{\left\Vert \Delta ^{\frac{s}{2}}\left(
\varphi v_{i}\right) \right\Vert _{L^{\frac{n}{s}}}^{\frac{n}{n-s}}}}d\mu \\
&\leq &c\left( M,g\right) .
\end{eqnarray*}%
Next we observe that%
\begin{eqnarray*}
&&\left\vert u_{i}\right\vert ^{\frac{n}{n-s}} \\
&=&\left\vert \left( v_{i}-\overline{\varphi v_{i}}\right) +u+\overline{%
\varphi v_{i}}\right\vert ^{\frac{n}{n-s}} \\
&\leq &\left( 1+\varepsilon \right) \left\vert v_{i}-\overline{\varphi v_{i}}%
\right\vert ^{\frac{n}{n-s}}+c\left( \varepsilon \right) \left\vert
u\right\vert ^{\frac{n}{n-s}}+c\left( \varepsilon \right) \left\vert 
\overline{\varphi v_{i}}\right\vert ^{\frac{n}{n-s}},
\end{eqnarray*}%
hence%
\begin{equation*}
e^{a_{s,n}\left\vert u_{i}\right\vert ^{\frac{n}{n-s}}}\leq e^{\left(
1+\varepsilon \right) a_{s,n}\left\vert v_{i}-\overline{\varphi v_{i}}%
\right\vert ^{\frac{n}{n-s}}}e^{c\left( \varepsilon \right) \left\vert
u\right\vert ^{\frac{n}{n-s}}}e^{c\left( \varepsilon \right) \left\vert 
\overline{\varphi v_{i}}\right\vert ^{\frac{n}{n-s}}}.
\end{equation*}%
Since $e^{\left( 1+\varepsilon \right) a_{s,n}\left\vert v_{i}-\overline{%
\varphi v_{i}}\right\vert ^{\frac{n}{n-s}}}$ is bounded in $L^{\frac{p_{1}}{%
1+\varepsilon }}\left( B_{r}\left( x\right) \right) $, $e^{c\left(
\varepsilon \right) \left\vert u\right\vert ^{\frac{n}{n-s}}}\in L^{q}\left(
B_{r}\left( x\right) \right) $ for any $0<q<\infty $ (by Lemma \ref{lem4.1}%
), $e^{c\left( \varepsilon \right) \left\vert \overline{\varphi v_{i}}%
\right\vert ^{\frac{n}{n-s}}}\rightarrow 1$ as $i\rightarrow \infty $ and $%
\frac{p_{1}}{1+\varepsilon }>p$, it follows from Holder inequality that $%
e^{a_{s,n}\left\vert u_{i}\right\vert ^{\frac{n}{n-s}}}$ is bounded in $%
L^{p}\left( B_{r}\left( x\right) \right) $.
\end{proof}

Proposition \ref{prop4.1} together with a covering argument implies the
following corollary, just like the proof of Corollary \ref{cor2.1}.

\begin{corollary}
\label{cor4.1}Assume $u_{i}\in W^{s,\frac{n}{s}}\left( M\right) $ such that $%
\overline{u_{i}}=0$ and $\left\Vert \Delta ^{\frac{s}{2}}u_{i}\right\Vert
_{L^{\frac{n}{s}}}\leq 1$. We also assume $u_{i}\rightharpoonup u$ weakly in 
$W^{s,\frac{n}{s}}\left( M\right) $ and%
\begin{equation}
\left\vert \Delta ^{\frac{s}{2}}u_{i}\right\vert ^{\frac{n}{s}}d\mu
\rightarrow \left\vert \Delta ^{\frac{s}{2}}u\right\vert ^{\frac{n}{s}}d\mu
+\sigma  \label{eq4.16}
\end{equation}%
as measure. Let $K$ be a compact subset of $M$ and%
\begin{equation}
\kappa =\max_{x\in K}\sigma \left( \left\{ x\right\} \right) .
\label{eq4.17}
\end{equation}

\begin{enumerate}
\item If $\kappa <1$, then for any $1\leq p<\kappa ^{-\frac{s}{n-s}}$, 
\begin{equation}
\sup_{i}\int_{K}e^{a_{s,n}p\left\vert u_{i}\right\vert ^{\frac{n}{n-s}}}d\mu
<\infty .  \label{eq4.18}
\end{equation}

\item If $\kappa =1$, then $\sigma =\delta _{x_{0}}$ for some $x_{0}\in K$, $%
u=0$ and after passing to a subsequence,%
\begin{equation}
e^{a_{s,n}\left\vert u_{i}\right\vert ^{\frac{n}{n-s}}}\rightarrow
1+c_{0}\delta _{x_{0}}  \label{eq4.19}
\end{equation}%
as measure for some $c_{0}\geq 0$.
\end{enumerate}
\end{corollary}

We are ready to derive an analogy of Theorem \ref{thm2.1} for $W^{s,\frac{n}{%
s}}\left( M\right) $.

\begin{theorem}
\label{thm4.1}Let $s\in \mathbb{N}$ be an even number strictly less than $n$%
. Assume $\alpha >0$, $m_{i}>0$, $m_{i}\rightarrow \infty $, $u_{i}\in W^{s,%
\frac{n}{s}}\left( M^{n}\right) $ such that $\overline{u_{i}}=0$ and%
\begin{equation}
\log \int_{M}e^{nm_{i}u_{i}}d\mu \geq \alpha m_{i}^{\frac{n}{s}}.
\label{eq4.20}
\end{equation}%
We also assume $u_{i}\rightharpoonup u$ weakly in $W^{s,\frac{n}{s}}\left(
M\right) $, $\left\vert \Delta ^{\frac{s}{2}}u_{i}\right\vert ^{\frac{n}{s}%
}d\mu \rightarrow \left\vert \Delta ^{\frac{s}{2}}u\right\vert ^{\frac{n}{s}%
}d\mu +\sigma $ as measure and%
\begin{equation}
\frac{e^{nm_{i}u_{i}}}{\int_{M}e^{nm_{i}u_{i}}d\mu }\rightarrow \nu
\label{eq4.21}
\end{equation}%
as measure. Let%
\begin{equation}
\left\{ x\in M:\sigma \left( \left\{ x\right\} \right) \geq \alpha
_{s,n}^{-1}\alpha \right\} =\left\{ x_{1},\cdots ,x_{N}\right\} ,
\label{eq4.22}
\end{equation}%
here%
\begin{equation}
\alpha _{s,n}=s\left( \frac{n-s}{n}\left\vert \mathbb{S}^{n-1}\right\vert
\right) ^{\frac{n-s}{s}}\left( \frac{\Gamma \left( \frac{n-s}{2}\right) }{%
\pi ^{\frac{n}{2}}2^{s}\Gamma \left( \frac{s}{2}\right) }\right) ^{\frac{n}{s%
}},  \label{eq4.23}
\end{equation}%
then%
\begin{equation}
\nu =\sum_{i=1}^{N}\nu _{i}\delta _{x_{i}},  \label{eq4.24}
\end{equation}%
here $\nu _{i}\geq 0$ and $\sum_{i=1}^{N}\nu _{i}=1$.
\end{theorem}

\begin{proof}
Assume $x\in M$ such that $\sigma \left( \left\{ x\right\} \right) <\alpha
_{s,n}^{-1}\alpha $, then we claim for some $r>0$, $\nu \left( B_{r}\left(
x\right) \right) =0$. Indeed we fix $p$ such that%
\begin{equation}
\alpha _{s,n}^{\frac{s}{n-s}}\alpha ^{-\frac{s}{n-s}}<p<\sigma \left(
\left\{ x\right\} \right) ^{-\frac{s}{n-s}},  \label{eq4.25}
\end{equation}%
it follows from Proposition \ref{prop4.1} that for some $r>0$,%
\begin{equation}
\int_{B_{r}\left( x\right) }e^{a_{s,n}p\left\vert u_{i}\right\vert ^{\frac{n%
}{n-s}}}d\mu \leq c,  \label{eq4.26}
\end{equation}%
here $c$ is a positive constant independent of $i$. By Young's inequality we
have%
\begin{equation}
nm_{i}u_{i}\leq a_{s,n}p\left\vert u_{i}\right\vert ^{\frac{n}{n-s}}+\frac{%
\alpha _{s,n}m_{i}^{\frac{n}{s}}}{p^{\frac{n-s}{s}}},  \label{eq4.27}
\end{equation}%
hence%
\begin{equation*}
\int_{B_{r}\left( x\right) }e^{nm_{i}u_{i}}d\mu \leq ce^{\frac{\alpha
_{s,n}m_{i}^{\frac{n}{s}}}{p^{\frac{n-s}{s}}}}.
\end{equation*}%
It follows that%
\begin{equation*}
\frac{\int_{B_{r}\left( x\right) }e^{nm_{i}u_{i}}d\mu }{%
\int_{M}e^{nm_{i}u_{i}}d\mu }\leq ce^{\left( \frac{\alpha _{s,n}}{p^{\frac{%
n-s}{s}}}-\alpha \right) m_{i}^{\frac{n}{s}}}.
\end{equation*}%
In particular,%
\begin{equation*}
\nu \left( B_{r}\left( x\right) \right) \leq \lim \inf_{i\rightarrow \infty }%
\frac{\int_{B_{r}\left( x\right) }e^{nm_{i}u_{i}}d\mu }{%
\int_{M}e^{nm_{i}u_{i}}d\mu }=0.
\end{equation*}%
We get $\nu \left( B_{r}\left( x\right) \right) =0$. The claim is proved.

Clearly the claim implies%
\begin{equation}
\nu \left( M\backslash \left\{ x_{1},\cdots ,x_{N}\right\} \right) =0.
\label{eq4.28}
\end{equation}%
Hence $\nu =\sum_{i=1}^{N}\nu _{i}\delta _{x_{i}},$ with $\nu _{i}\geq 0$
and $\sum_{i=1}^{N}\nu _{i}=1$.
\end{proof}

The analogy of Theorem \ref{thm3.2} is the following

\begin{theorem}
\label{thm4.2}Let $s\in \mathbb{N}$ be an even number strictly less than $n$%
. If $u\in W^{s,\frac{n}{s}}\left( \mathbb{S}^{n}\right) $ such that $\int_{%
\mathbb{S}^{n}}ud\mu =0$ (here $\mu $ is the standard measure on $\mathbb{S}%
^{n}$) and for every $p\in \overset{\circ }{\mathcal{P}}_{m}$,%
\begin{equation}
\left\vert \int_{\mathbb{S}^{n}}pe^{nu}d\mu \right\vert \leq b\left(
p\right) ,  \label{eq4.29}
\end{equation}%
here $b\left( p\right) $ is a nonnegative number depending only on $p$, then
for any $\varepsilon >0$,%
\begin{equation}
\log \int_{\mathbb{S}^{n}}e^{nu}d\mu \leq \left( \frac{\alpha _{s,n}}{%
N_{m}\left( \mathbb{S}^{n}\right) }+\varepsilon \right) \left\Vert \Delta ^{%
\frac{s}{2}}u\right\Vert _{L^{\frac{n}{s}}}^{\frac{n}{s}}+c\left(
m,n,b,\varepsilon \right) .  \label{eq4.30}
\end{equation}%
Here $N_{m}\left( \mathbb{S}^{n}\right) $ is defined in Definition \ref%
{def3.1} and%
\begin{equation}
\alpha _{s,n}=s\left( \frac{n-s}{n}\left\vert \mathbb{S}^{n-1}\right\vert
\right) ^{\frac{n-s}{s}}\left( \frac{\Gamma \left( \frac{n-s}{2}\right) }{%
\pi ^{\frac{n}{2}}2^{s}\Gamma \left( \frac{s}{2}\right) }\right) ^{\frac{n}{s%
}}.  \label{eq4.31}
\end{equation}
\end{theorem}

\begin{proof}
Let $\alpha =\frac{\alpha _{s,n}}{N_{m}\left( \mathbb{S}^{n}\right) }%
+\varepsilon $. If the inequality (\ref{eq4.30}) is not true, then there
exists $v_{i}\in W^{s,\frac{n}{s}}\left( \mathbb{S}^{n}\right) $ such that $%
\overline{v_{i}}=0$, 
\begin{equation*}
\left\vert \int_{\mathbb{S}^{n}}pe^{nv_{i}}d\mu \right\vert \leq b\left(
p\right) ,
\end{equation*}%
for all $p\in \overset{\circ }{\mathcal{P}}_{m}$ and%
\begin{equation*}
\log \int_{\mathbb{S}^{n}}e^{nv_{i}}d\mu -\alpha \left\Vert \Delta ^{\frac{s%
}{2}}v_{i}\right\Vert _{L^{\frac{n}{s}}}^{\frac{n}{s}}\rightarrow \infty
\end{equation*}%
as $i\rightarrow \infty $. In particular $\int_{\mathbb{S}%
^{n}}e^{nv_{i}}d\mu \rightarrow \infty $. By (\ref{eq4.7}),%
\begin{equation*}
\log \int_{\mathbb{S}^{n}}e^{nv_{i}}d\mu \leq \alpha _{s,n}\left\Vert \Delta
^{\frac{s}{2}}v_{i}\right\Vert _{L^{\frac{n}{s}}}^{\frac{n}{s}}+c\left(
n\right) ,
\end{equation*}%
hence $\left\Vert \Delta ^{\frac{s}{2}}v_{i}\right\Vert _{L^{\frac{n}{s}%
}}\rightarrow \infty $. Let $m_{i}=\left\Vert \Delta ^{\frac{s}{2}%
}v_{i}\right\Vert _{L^{\frac{n}{s}}}$ and $u_{i}=\frac{v_{i}}{m_{i}}$, then $%
m_{i}\rightarrow \infty $, $\left\Vert \Delta ^{\frac{s}{2}}u_{i}\right\Vert
_{L^{\frac{n}{s}}}=1$ and $\overline{u_{i}}=0$. After passing to a
subsequence, we have%
\begin{eqnarray*}
u_{i} &\rightharpoonup &u\text{ weakly in }W^{s,\frac{n}{s}}\left( \mathbb{S}%
^{n}\right) ; \\
\log \int_{\mathbb{S}^{n}}e^{nm_{i}u_{i}}d\mu -\alpha m_{i}^{\frac{n}{s}}
&\rightarrow &\infty ; \\
\left\vert \Delta ^{\frac{s}{2}}u_{i}\right\vert ^{\frac{n}{s}}d\mu
&\rightarrow &\left\vert \Delta ^{\frac{s}{2}}u\right\vert ^{\frac{n}{s}%
}d\mu +\sigma \text{ as measure;} \\
\frac{e^{nm_{i}u_{i}}}{\int_{\mathbb{S}^{n}}e^{nm_{i}u_{i}}d\mu }
&\rightarrow &\nu \text{ as measure.}
\end{eqnarray*}%
Let%
\begin{equation}
\left\{ x\in \mathbb{S}^{n}:\sigma \left( \left\{ x\right\} \right) \geq
\alpha _{s,n}^{-1}\alpha \right\} =\left\{ x_{1},\cdots ,x_{N}\right\} ,
\label{eq4.32}
\end{equation}%
then it follows from Theorem \ref{thm4.1} that%
\begin{equation}
\nu =\sum_{i=1}^{N}\nu _{i}\delta _{x_{i}}.  \label{eq4.33}
\end{equation}%
Moreover $\nu _{i}\geq 0$, $\sum_{i=1}^{N}\nu _{i}=1$ and%
\begin{equation*}
\sum_{i=1}^{N}\nu _{i}p\left( x_{i}\right) =0
\end{equation*}%
for any $p\in \overset{\circ }{\mathcal{P}}_{m}$. It follows that%
\begin{equation*}
\alpha _{s,n}^{-1}\alpha N\leq 1,
\end{equation*}%
and $N\in \mathcal{N}_{m}\left( \mathbb{S}^{n}\right) $. We get%
\begin{equation*}
\alpha \leq \frac{\alpha _{s,n}}{N}\leq \frac{\alpha _{s,n}}{N_{m}\left( 
\mathbb{S}^{n}\right) }.
\end{equation*}%
This contradicts with the choice of $\alpha $.
\end{proof}

Similar as before, the constant $\frac{\alpha _{s,n}}{N_{m}\left( \mathbb{S}%
^{n}\right) }+\varepsilon $ is almost optimal. It is also interesting to
compare Theorem \ref{thm4.2} with \cite[lemma 4.3]{BCY} and \cite[lemma 4.6]%
{CY2}.

\subsection{Paneitz operator in dimension $4$\label{sec4.2}}

Let $\left( M^{4},g\right) $ be a smooth compact Riemannian manifold with
dimension $4$. The Paneitz operator is given by (see \cite{CY2, HY})%
\begin{equation}
Pu=\Delta ^{2}u+2\func{div}\left( Rc\left( \nabla u,e_{i}\right)
e_{i}\right) -\frac{2}{3}\func{div}\left( R\nabla u\right) .  \label{eq4.34}
\end{equation}%
Here $e_{1},e_{2},e_{3},e_{4}$ is a local orthonormal frame with respect to $%
g$. The associated $Q$ curvature is%
\begin{equation}
Q=-\frac{1}{6}\Delta R-\frac{1}{2}\left\vert Rc\right\vert ^{2}+\frac{1}{6}%
R^{2}.  \label{eq4.35}
\end{equation}%
In 4-dimensional conformal geometry, $P$ and $Q$ play the same roles as $%
-\Delta $ and Gauss curvature in 2-dimensional conformal geometry.

For $u\in C^{\infty }\left( M\right) $, let%
\begin{eqnarray}
E\left( u\right) &=&\int_{M}Pu\cdot ud\mu  \label{eq4.36} \\
&=&\int_{M}\left( \left( \Delta u\right) ^{2}-2Rc\left( \nabla u,\nabla
u\right) +\frac{2}{3}R\left\vert \nabla u\right\vert ^{2}\right) d\mu . 
\notag
\end{eqnarray}%
By this formula, we know $E\left( u\right) $ still makes sense for $u\in
H^{2}\left( M\right) =W^{2,2}\left( M\right) $.

In \cite{CY2}, it is shown that if $P\geq 0$ and $\ker P=\left\{ \text{%
constant functions}\right\} $, then for any $u\in H^{2}\left( M\right)
\backslash \left\{ 0\right\} $ with $\overline{u}=0$,%
\begin{equation}
\int_{M}e^{32\pi ^{2}\frac{u^{2}}{E\left( u\right) }}d\mu \leq c\left(
M,g\right) .  \label{eq4.37}
\end{equation}%
In particular%
\begin{equation}
\log \int_{M}e^{4u}d\mu \leq \frac{1}{8\pi ^{2}}E\left( u\right)
+c_{1}\left( M,g\right) .  \label{eq4.38}
\end{equation}%
Note that $32\pi ^{2}=a_{2,4}$ and $\frac{1}{8\pi ^{2}}=\alpha _{2,4}$.

We want to remark that if $P\geq 0$ and $\ker P=\left\{ \text{constant
functions}\right\} $, then for any $u\in H^{2}\left( M\right) $ with $%
\overline{u}=0$,%
\begin{equation}
\left\Vert u\right\Vert _{L^{2}}^{2}\leq c\left( M,g\right) E\left( u\right)
.  \label{eq4.39}
\end{equation}%
On the other hand, it follows from standard elliptic theory (or the
Bochner's identity) that%
\begin{eqnarray*}
\left\Vert u\right\Vert _{H^{2}}^{2} &\leq &c\left( M,g\right) \left(
\left\Vert \Delta u\right\Vert _{L^{2}}^{2}+\left\Vert u\right\Vert
_{L^{2}}^{2}\right) \\
&\leq &c\left( M,g\right) \left( E\left( u\right) +\left\Vert u\right\Vert
_{H^{1}}^{2}\right) \\
&\leq &c\left( M,g\right) E\left( u\right) +\frac{1}{2}\left\Vert
u\right\Vert _{H^{2}}^{2}+c\left( M,g\right) \left\Vert u\right\Vert
_{L^{2}}^{2} \\
&\leq &\frac{1}{2}\left\Vert u\right\Vert _{H^{2}}^{2}+c\left( M,g\right)
E\left( u\right) .
\end{eqnarray*}%
We have used the interpolation inequality in between. It follows that%
\begin{equation}
\left\Vert u\right\Vert _{H^{2}}^{2}\leq c\left( M,g\right) E\left( u\right)
.  \label{eq4.40}
\end{equation}

It is also worth pointing out that on the standard $\mathbb{S}^{4}$, $P\geq
0 $ and $\ker P=\left\{ \text{constant functions}\right\} $. Moreover in 
\cite{Gur, GurV}, some general criterion for such positivity condition to be
valid were derived.

\begin{proposition}
\label{prop4.2}Let $\left( M^{4},g\right) $ be a smooth compact Riemannian
manifold with $P\geq 0$ and $\ker P=\left\{ \text{constant functions}%
\right\} $. Assume $u_{i}\in H^{2}\left( M\right) $ such that $\overline{%
u_{i}}=0$ and $E\left( u_{i}\right) \leq 1$. We also assume $%
u_{i}\rightharpoonup u$ weakly in $H^{2}\left( M\right) $ and%
\begin{equation}
\left( \Delta u_{i}\right) ^{2}d\mu \rightarrow \left( \Delta u\right)
^{2}d\mu +\sigma  \label{eq4.41}
\end{equation}%
as measure. Let $K$ be a compact subset of $M$ and%
\begin{equation}
\kappa =\max_{x\in K}\sigma \left( \left\{ x\right\} \right) .
\label{eq4.42}
\end{equation}

\begin{enumerate}
\item If $\kappa <1$, then for any $1\leq p<\frac{1}{\kappa }$, 
\begin{equation}
\sup_{i}\int_{K}e^{32\pi ^{2}pu_{i}^{2}}d\mu <\infty .  \label{eq4.43}
\end{equation}

\item If $\kappa =1$, then $\sigma =\delta _{x_{0}}$ for some $x_{0}\in K$, $%
u=0$ and after passing to a subsequence,%
\begin{equation}
e^{32\pi ^{2}u_{i}^{2}}\rightarrow 1+c_{0}\delta _{x_{0}}  \label{eq4.44}
\end{equation}%
as measure for some $c_{0}\geq 0$.
\end{enumerate}
\end{proposition}

\begin{proof}
Since $u_{i}\rightharpoonup u$ weakly in $H^{2}\left( M\right) $, we see $%
u_{i}\rightarrow u$ in $H^{1}\left( M\right) $. It follows that $\overline{u}%
=0$ and%
\begin{equation*}
E\left( u_{i}\right) \rightarrow E\left( u\right) +\sigma \left( M\right) .
\end{equation*}%
Since $E\left( u_{i}\right) \leq 1$, we get%
\begin{equation*}
E\left( u\right) +\sigma \left( M\right) \leq 1.
\end{equation*}%
By the assumption on $P$ we know $E\left( u\right) \geq 0$ and $E\left(
u\right) =0$ if and only if $u=0$, hence $\sigma \left( M\right) \leq 1$.
With these facts at hand, Proposition \ref{prop4.2} follows from Proposition %
\ref{prop4.1} by the same argument as in the proof of Corollary \ref{cor2.1}.
\end{proof}

On the standard $\mathbb{S}^{4}$, we have $P=\Delta ^{2}-2\Delta $ and%
\begin{equation*}
E\left( u\right) =\int_{\mathbb{S}^{4}}\left( \left( \Delta u\right)
^{2}+2\left\vert \nabla u\right\vert ^{2}\right) d\mu \geq \left\Vert \Delta
u\right\Vert _{L^{2}}^{2}
\end{equation*}%
for $u\in H^{2}\left( \mathbb{S}^{4}\right) $. It follows from this
inequality and Theorem \ref{thm4.2} that

\begin{proposition}
\label{prop4.3}If $u\in H^{2}\left( \mathbb{S}^{4}\right) $ such that $\int_{%
\mathbb{S}^{4}}ud\mu =0$ (here $\mu $ is the standard measure on $\mathbb{S}%
^{4}$) and for every $p\in \overset{\circ }{\mathcal{P}}_{m}$,%
\begin{equation}
\left\vert \int_{\mathbb{S}^{4}}pe^{nu}d\mu \right\vert \leq b\left(
p\right) ,  \label{eq4.45}
\end{equation}%
here $b\left( p\right) $ is a nonnegative number depending only on $p$, then
for any $\varepsilon >0$,%
\begin{equation}
\log \int_{\mathbb{S}^{4}}e^{4u}d\mu \leq \left( \frac{1}{8\pi
^{2}N_{m}\left( \mathbb{S}^{4}\right) }+\varepsilon \right) \int_{\mathbb{S}%
^{4}}\left( \left( \Delta u\right) ^{2}+2\left\vert \nabla u\right\vert
^{2}\right) d\mu +c\left( m,b,\varepsilon \right) .  \label{eq4.46}
\end{equation}%
Here $N_{m}\left( \mathbb{S}^{4}\right) $ is defined in Definition \ref%
{def3.1}.
\end{proposition}

In view of Proposition \ref{prop4.3} and the recent proof of sharp version
of Aubin's Moser-Trudinger inequality on $\mathbb{S}^{2}$ in \cite{GuM}, it
is tempting to conjecture that for $u\in H^{2}\left( \mathbb{S}^{4}\right) $
with $\int_{\mathbb{S}^{4}}ud\mu =0$ and%
\begin{equation}
\int_{\mathbb{S}^{4}}x_{i}e^{4u}d\mu \left( x\right) =0\text{ for }1\leq
i\leq 5\text{,}  \label{eq4.47}
\end{equation}%
we have%
\begin{equation}
\log \left( \frac{1}{\left\vert \mathbb{S}^{4}\right\vert }\int_{\mathbb{S}%
^{4}}e^{4u}d\mu \right) \leq \frac{1}{16\pi ^{2}}\int_{\mathbb{S}^{4}}\left(
\left( \Delta u\right) ^{2}+2\left\vert \nabla u\right\vert ^{2}\right) d\mu
.  \label{eq4.48}
\end{equation}%
In a recent work \cite{Gu}, some progress has been made toward proving (\ref%
{eq4.48}) for axially symmetric functions.

\subsection{$W^{s,\frac{n}{s}}\left( M^{n}\right) $ for odd $s$\label{sec4.3}%
}

Let $s\in \mathbb{N}$ be an odd number strictly less than $n$. Denote%
\begin{equation}
a_{s,n}=\frac{n}{\left\vert \mathbb{S}^{n-1}\right\vert }\left( \frac{\pi ^{%
\frac{n}{2}}2^{s}\Gamma \left( \frac{s+1}{2}\right) }{\Gamma \left( \frac{%
n-s+1}{2}\right) }\right) ^{\frac{n}{n-s}}.  \label{eq4.49}
\end{equation}%
The Moser-Trudinger inequality (see \cite{F}) tells us that for every $u\in
W^{s,\frac{n}{s}}\left( M\right) \backslash \left\{ 0\right\} $ with $%
\overline{u}=0$,%
\begin{equation}
\int_{M}\exp \left( a_{s,n}\frac{\left\vert u\right\vert ^{\frac{n}{n-s}}}{%
\left\Vert \nabla \Delta ^{\frac{s-1}{2}}u\right\Vert _{L^{\frac{n}{s}}}^{%
\frac{n}{n-s}}}\right) d\mu \leq c\left( M,g\right) .  \label{eq4.50}
\end{equation}%
This implies the Moser-Trudinger-Onofri inequality%
\begin{equation}
\log \int_{M}e^{nu}d\mu \leq \alpha _{s,n}\left\Vert \nabla \Delta ^{\frac{%
s-1}{2}}u\right\Vert _{L^{\frac{n}{s}}}^{\frac{n}{s}}+c_{1}\left( M,g\right)
.  \label{eq4.51}
\end{equation}%
Here%
\begin{equation}
\alpha _{s,n}=s\left( \frac{n-s}{n}\left\vert \mathbb{S}^{n-1}\right\vert
\right) ^{\frac{n-s}{s}}\left( \frac{\Gamma \left( \frac{n-s+1}{2}\right) }{%
\pi ^{\frac{n}{2}}2^{s}\Gamma \left( \frac{s+1}{2}\right) }\right) ^{\frac{n%
}{s}}.  \label{eq4.52}
\end{equation}

\begin{proposition}
\label{prop4.4}Assume $s\in \mathbb{N}$ is an odd number strictly less than $%
n$, $u_{i}\in W^{s,\frac{n}{s}}\left( M^{n}\right) $ such that $\overline{%
u_{i}}=0$, $u_{i}\rightharpoonup u$ weakly in $W^{s,\frac{n}{s}}\left(
M\right) $ and%
\begin{equation}
\left\vert \nabla \Delta ^{\frac{s-1}{2}}u_{i}\right\vert ^{\frac{n}{s}}d\mu
\rightarrow \left\vert \nabla \Delta ^{\frac{s-1}{2}}u\right\vert ^{\frac{n}{%
s}}d\mu +\sigma  \label{eq4.53}
\end{equation}%
as measure. If $x\in M$ and $p\in \mathbb{R}$ such that $0<p<\sigma \left(
\left\{ x\right\} \right) ^{-\frac{s}{n-s}}$, then for some $r>0$,%
\begin{equation}
\sup_{i}\int_{B_{r}\left( x\right) }e^{a_{s,n}p\left\vert u_{i}\right\vert ^{%
\frac{n}{n-s}}}d\mu <\infty .  \label{eq4.54}
\end{equation}%
Here%
\begin{equation}
a_{s,n}=\frac{n}{\left\vert \mathbb{S}^{n-1}\right\vert }\left( \frac{\pi ^{%
\frac{n}{2}}2^{s}\Gamma \left( \frac{s+1}{2}\right) }{\Gamma \left( \frac{%
n-s+1}{2}\right) }\right) ^{\frac{n}{n-s}}.  \label{eq4.55}
\end{equation}
\end{proposition}

Since the proof of Proposition \ref{prop4.4} is almost identical to the
proof of Proposition \ref{prop4.1}, we omit it here. The same will happen to
Corollary \ref{cor4.2} and Theorems \ref{thm4.3}, \ref{thm4.4} below.

\begin{corollary}
\label{cor4.2}Assume $u_{i}\in W^{s,\frac{n}{s}}\left( M\right) $ such that $%
\overline{u_{i}}=0$ and $\left\Vert \nabla \Delta ^{\frac{s-1}{2}%
}u_{i}\right\Vert _{L^{\frac{n}{s}}}\leq 1$. We also assume $%
u_{i}\rightharpoonup u$ weakly in $W^{s,\frac{n}{s}}\left( M\right) $ and%
\begin{equation}
\left\vert \nabla \Delta ^{\frac{s-1}{2}}u_{i}\right\vert ^{\frac{n}{s}}d\mu
\rightarrow \left\vert \nabla \Delta ^{\frac{s-1}{2}}u\right\vert ^{\frac{n}{%
s}}d\mu +\sigma  \label{eq4.56}
\end{equation}%
as measure. Let $K$ be a compact subset of $M$ and%
\begin{equation}
\kappa =\max_{x\in K}\sigma \left( \left\{ x\right\} \right) .
\label{eq4.57}
\end{equation}

\begin{enumerate}
\item If $\kappa <1$, then for any $1\leq p<\kappa ^{-\frac{s}{n-s}}$, 
\begin{equation}
\sup_{i}\int_{K}e^{a_{s,n}p\left\vert u_{i}\right\vert ^{\frac{n}{n-s}}}d\mu
<\infty .  \label{eq4.58}
\end{equation}

\item If $\kappa =1$, then $\sigma =\delta _{x_{0}}$ for some $x_{0}\in K$, $%
u=0$ and after passing to a subsequence,%
\begin{equation}
e^{a_{s,n}\left\vert u_{i}\right\vert ^{\frac{n}{n-s}}}\rightarrow
1+c_{0}\delta _{x_{0}}  \label{eq4.59}
\end{equation}%
as measure for some $c_{0}\geq 0$.
\end{enumerate}
\end{corollary}

\begin{theorem}
\label{thm4.3}Let $s\in \mathbb{N}$ be an odd number strictly less than $n$.
Assume $\alpha >0$, $m_{i}>0$, $m_{i}\rightarrow \infty $, $u_{i}\in W^{s,%
\frac{n}{s}}\left( M^{n}\right) $ such that $\overline{u_{i}}=0$ and%
\begin{equation}
\log \int_{M}e^{nm_{i}u_{i}}d\mu \geq \alpha m_{i}^{\frac{n}{s}}.
\label{eq4.60}
\end{equation}%
We also assume $u_{i}\rightharpoonup u$ weakly in $W^{s,\frac{n}{s}}\left(
M\right) $, $\left\vert \nabla \Delta ^{\frac{s-1}{2}}u_{i}\right\vert ^{%
\frac{n}{s}}d\mu \rightarrow \left\vert \nabla \Delta ^{\frac{s-1}{2}%
}u\right\vert ^{\frac{n}{s}}d\mu +\sigma $ as measure and%
\begin{equation}
\frac{e^{nm_{i}u_{i}}}{\int_{M}e^{nm_{i}u_{i}}d\mu }\rightarrow \nu
\label{eq4.61}
\end{equation}%
as measure. Let%
\begin{equation}
\left\{ x\in M:\sigma \left( \left\{ x\right\} \right) \geq \alpha
_{s,n}^{-1}\alpha \right\} =\left\{ x_{1},\cdots ,x_{N}\right\} ,
\label{eq4.62}
\end{equation}%
here%
\begin{equation}
\alpha _{s,n}=s\left( \frac{n-s}{n}\left\vert \mathbb{S}^{n-1}\right\vert
\right) ^{\frac{n-s}{s}}\left( \frac{\Gamma \left( \frac{n-s+1}{2}\right) }{%
\pi ^{\frac{n}{2}}2^{s}\Gamma \left( \frac{s+1}{2}\right) }\right) ^{\frac{n%
}{s}},  \label{eq4.63}
\end{equation}%
then%
\begin{equation}
\nu =\sum_{i=1}^{N}\nu _{i}\delta _{x_{i}},  \label{eq4.64}
\end{equation}%
here $\nu _{i}\geq 0$ and $\sum_{i=1}^{N}\nu _{i}=1$.
\end{theorem}

\begin{theorem}
\label{thm4.4}Let $s\in \mathbb{N}$ be an odd number strictly less than $n$.
If $u\in W^{s,\frac{n}{s}}\left( \mathbb{S}^{n}\right) $ such that $\int_{%
\mathbb{S}^{n}}ud\mu =0$ (here $\mu $ is the standard measure on $\mathbb{S}%
^{n}$) and for every $p\in \overset{\circ }{\mathcal{P}}_{m}$,%
\begin{equation}
\left\vert \int_{\mathbb{S}^{n}}pe^{nu}d\mu \right\vert \leq b\left(
p\right) ,  \label{eq4.65}
\end{equation}%
here $b\left( p\right) $ is a nonnegative number depending only on $p$, then
for any $\varepsilon >0$,%
\begin{equation}
\log \int_{\mathbb{S}^{n}}e^{nu}d\mu \leq \left( \frac{\alpha _{s,n}}{%
N_{m}\left( \mathbb{S}^{n}\right) }+\varepsilon \right) \left\Vert \nabla
\Delta ^{\frac{s-1}{2}}u\right\Vert _{L^{\frac{n}{s}}}^{\frac{n}{s}}+c\left(
m,n,b,\varepsilon \right) .  \label{eq4.66}
\end{equation}%
Here $N_{m}\left( \mathbb{S}^{n}\right) $ is defined in Definition \ref%
{def3.1} and%
\begin{equation}
\alpha _{s,n}=s\left( \frac{n-s}{n}\left\vert \mathbb{S}^{n-1}\right\vert
\right) ^{\frac{n-s}{s}}\left( \frac{\Gamma \left( \frac{n-s+1}{2}\right) }{%
\pi ^{\frac{n}{2}}2^{s}\Gamma \left( \frac{s+1}{2}\right) }\right) ^{\frac{n%
}{s}}.  \label{eq4.67}
\end{equation}
\end{theorem}

\subsection{Functions on compact surfaces with boundary\label{sec4.4}}

In this subsection, we assume $\left( M^{2},g\right) $ is a smooth compact
surface with nonempty boundary and a smooth Riemannian metric $g$.

\subsubsection{Functions with zero boundary value}

We denote $H_{0}^{1}\left( M\right) =W_{0}^{1,2}\left( M\right) $. It
follows from \cite{CY1} that for every $u\in H_{0}^{1}\left( M\right)
\backslash \left\{ 0\right\} $,%
\begin{equation}
\int_{M}e^{4\pi \frac{u^{2}}{\left\Vert \nabla u\right\Vert _{L^{2}}^{2}}%
}d\mu \leq c\left( M,g\right) .  \label{eq4.68}
\end{equation}%
As a consequence,%
\begin{equation}
\log \int_{M}e^{2u}d\mu \leq \frac{1}{4\pi }\left\Vert \nabla u\right\Vert
_{L^{2}}^{2}+c_{1}\left( M,g\right) .  \label{eq4.69}
\end{equation}

\begin{lemma}
\label{lem4.2}For $u\in H_{0}^{1}\left( M\right) $ and $a>0$, we have%
\begin{equation}
\int_{M}e^{au^{2}}d\mu <\infty .  \label{eq4.70}
\end{equation}
\end{lemma}

\begin{proof}
Indeed, fix $\varepsilon >0$ tiny, we can find $v\in C_{c}^{\infty }\left(
M\right) $ such that $\left\Vert \nabla u-\nabla v\right\Vert
_{L^{2}}<\varepsilon $. With this $v$ and (\ref{eq4.68}) at hands, we can
proceed exactly the same way as in the proof of Lemma \ref{lem2.1}.
\end{proof}

\begin{proposition}
\label{prop4.5}Assume $u_{i}\in H_{0}^{1}\left( M\right) $ such that $%
u_{i}\rightharpoonup u$ weakly in $H_{0}^{1}\left( M\right) $, and%
\begin{equation}
\left\vert \nabla u_{i}\right\vert ^{2}d\mu \rightarrow \left\vert \nabla
u\right\vert ^{2}d\mu +\sigma  \label{eq4.71}
\end{equation}%
as measure. If $x\in M$, $p\in \mathbb{R}$ satisfies $0<p<\frac{1}{\sigma
\left( \left\{ x\right\} \right) }$, then for some $r>0$,%
\begin{equation}
\sup_{i}\int_{B_{r}\left( x\right) }e^{4\pi pu_{i}^{2}}d\mu <\infty .
\label{eq4.72}
\end{equation}
\end{proposition}

\begin{proof}
Fix $p_{1}\in \left( p,\frac{1}{\sigma \left( \left\{ x\right\} \right) }%
\right) $, then%
\begin{equation}
\sigma \left( \left\{ x\right\} \right) <\frac{1}{p_{1}}.  \label{eq4.73}
\end{equation}%
We can find a $\varepsilon >0$ such that%
\begin{equation}
\left( 1+\varepsilon \right) \sigma \left( \left\{ x\right\} \right) <\frac{1%
}{p_{1}}  \label{eq4.74}
\end{equation}%
and%
\begin{equation}
\left( 1+\varepsilon \right) p<p_{1}.  \label{eq4.75}
\end{equation}%
Let $v_{i}=u_{i}-u$, then $v_{i}\rightharpoonup 0$ weakly in $%
H_{0}^{1}\left( M\right) $, $v_{i}\rightarrow 0$ in $L^{2}\left( M\right) $.
For any $\varphi \in C^{\infty }\left( M\right) $, we have%
\begin{eqnarray*}
&&\left\Vert \nabla \left( \varphi v_{i}\right) \right\Vert _{L^{2}}^{2} \\
&\leq &\left( \left\Vert \varphi \nabla v_{i}\right\Vert _{L^{2}}+\left\Vert
v_{i}\nabla \varphi \right\Vert _{L^{2}}\right) ^{2} \\
&\leq &\left( \left\Vert \varphi \nabla u_{i}\right\Vert _{L^{2}}+\left\Vert
\varphi \nabla u\right\Vert _{L^{2}}+\left\Vert v_{i}\nabla \varphi
\right\Vert _{L^{2}}\right) ^{2} \\
&\leq &\left( 1+\varepsilon \right) \left\Vert \varphi \nabla
u_{i}\right\Vert _{L^{2}}^{2}+c\left( \varepsilon \right) \left\Vert \varphi
\nabla u\right\Vert _{L^{n}}^{n}+c\left( \varepsilon \right) \left\Vert
v_{i}\nabla \varphi \right\Vert _{L^{n}}^{n}.
\end{eqnarray*}%
It follows that%
\begin{eqnarray*}
&&\lim \sup_{i\rightarrow \infty }\left\Vert \nabla \left( \varphi
v_{i}\right) \right\Vert _{L^{2}}^{2} \\
&\leq &\left( 1+\varepsilon \right) \left( \int_{M}\varphi ^{2}d\sigma
+\int_{M}\varphi ^{2}\left\vert \nabla u\right\vert ^{2}d\mu \right)
+c\left( \varepsilon \right) \left\Vert \varphi \nabla u\right\Vert
_{L^{2}}^{2} \\
&=&\left( 1+\varepsilon \right) \int_{M}\varphi ^{2}d\sigma +c_{1}\left(
\varepsilon \right) \int_{M}\varphi ^{2}\left\vert \nabla u\right\vert
^{2}d\mu .
\end{eqnarray*}%
We can find a $\varphi \in C^{\infty }\left( M\right) $ such that $\left.
\varphi \right\vert _{B_{r}\left( x\right) }=1$ for some $r>0$ and%
\begin{equation*}
\left( 1+\varepsilon \right) \int_{M}\varphi ^{2}d\sigma +c_{1}\left(
\varepsilon \right) \int_{M}\varphi ^{2}\left\vert \nabla u\right\vert
^{2}d\mu <\frac{1}{p_{1}}.
\end{equation*}%
Hence for $i$ large enough,%
\begin{equation*}
\left\Vert \nabla \left( \varphi v_{i}\right) \right\Vert _{L^{2}}^{2}<\frac{%
1}{p_{1}}.
\end{equation*}%
Note that $\varphi v_{i}\in H_{0}^{1}\left( M\right) $. We have%
\begin{eqnarray*}
\int_{B_{r}\left( x\right) }e^{4\pi p_{1}v_{i}^{2}}d\mu &\leq
&\int_{M}e^{4\pi p_{1}\left( \varphi v_{i}\right) ^{2}}d\mu \\
&\leq &\int_{M}e^{4\pi \frac{\left( \varphi v_{i}\right) ^{2}}{\left\Vert
\nabla \left( \varphi v_{i}\right) \right\Vert _{L^{2}}^{2}}}d\mu \\
&\leq &c\left( M,g\right) .
\end{eqnarray*}%
Next we observe that%
\begin{equation*}
u_{i}^{2}=\left( v_{i}+u\right) ^{2}\leq \left( 1+\varepsilon \right)
v_{i}^{2}+c\left( \varepsilon \right) u^{2},
\end{equation*}%
hence%
\begin{equation*}
e^{4\pi u_{i}^{2}}\leq e^{4\pi \left( 1+\varepsilon \right)
v_{i}^{2}}e^{c\left( \varepsilon \right) u^{2}}.
\end{equation*}%
Since $e^{4\pi \left( 1+\varepsilon \right) v_{i}^{2}}$ is bounded in $L^{%
\frac{p_{1}}{1+\varepsilon }}\left( B_{r}\left( x\right) \right) $, $%
e^{c\left( \varepsilon \right) u^{2}}\in L^{q}\left( B_{r}\left( x\right)
\right) $ for any $0<q<\infty $ (by Lemma \ref{lem4.2}), and $\frac{p_{1}}{%
1+\varepsilon }>p$, it follows from Holder inequality that $e^{4\pi
u_{i}^{2}}$ is bounded in $L^{p}\left( B_{r}\left( x\right) \right) $.
\end{proof}

\begin{corollary}
\label{cor4.3}Assume $u_{i}\in H_{0}^{1}\left( M\right) $ such that $%
\left\Vert \nabla u_{i}\right\Vert _{L^{2}}\leq 1$. We also assume $%
u_{i}\rightharpoonup u$ weakly in $H_{0}^{1}\left( M\right) $ and%
\begin{equation}
\left\vert \nabla u_{i}\right\vert ^{2}d\mu \rightarrow \left\vert \nabla
u\right\vert ^{2}d\mu +\sigma  \label{eq4.76}
\end{equation}%
as measure. Let $K$ be a compact subset of $M$ and%
\begin{equation}
\kappa =\max_{x\in K}\sigma \left( \left\{ x\right\} \right) .
\label{eq4.77}
\end{equation}

\begin{enumerate}
\item If $\kappa <1$, then for any $1\leq p<\frac{1}{\kappa }$, 
\begin{equation}
\sup_{i}\int_{K}e^{4\pi p\left\vert u_{i}\right\vert ^{\frac{n}{n-1}}}d\mu
<\infty .  \label{eq4.78}
\end{equation}

\item If $\kappa =1$, then $\sigma =\delta _{x_{0}}$ for some $x_{0}\in K$, $%
u=0$ and after passing to a subsequence,%
\begin{equation}
e^{4\pi u_{i}^{2}}\rightarrow 1+c_{0}\delta _{x_{0}}  \label{eq4.79}
\end{equation}%
as measure for some $c_{0}\geq 0$.
\end{enumerate}
\end{corollary}

With Proposition \ref{prop4.5}, we can derive Corollary \ref{cor4.3} exactly
in the same way as the proof of Corollary \ref{cor2.1}.

\begin{theorem}
\label{thm4.5}Let $\left( M^{2},g\right) $ be a smooth compact Riemann
surface with nonempty boundary. Assume $\alpha >0$, $m_{i}>0$, $%
m_{i}\rightarrow \infty $, $u_{i}\in H_{0}^{1}\left( M\right) $ and%
\begin{equation}
\log \int_{M}e^{2m_{i}u_{i}}d\mu \geq \alpha m_{i}^{2}.  \label{eq4.80}
\end{equation}%
We also assume $u_{i}\rightharpoonup u$ weakly in $H_{0}^{1}\left( M\right) $%
, $\left\vert \nabla u_{i}\right\vert ^{2}d\mu \rightarrow \left\vert \nabla
u\right\vert ^{2}d\mu +\sigma $ as measure and%
\begin{equation}
\frac{e^{2m_{i}u_{i}}}{\int_{M}e^{2m_{i}u_{i}}d\mu }\rightarrow \nu
\label{eq4.81}
\end{equation}%
as measure. Let%
\begin{equation}
\left\{ x\in M:\sigma \left( \left\{ x\right\} \right) \geq 4\pi \alpha
\right\} =\left\{ x_{1},\cdots ,x_{N}\right\} ,  \label{eq4.82}
\end{equation}%
then%
\begin{equation}
\nu =\sum_{i=1}^{N}\nu _{i}\delta _{x_{i}},  \label{eq4.83}
\end{equation}%
here $\nu _{i}\geq 0$ and $\sum_{i=1}^{N}\nu _{i}=1$.
\end{theorem}

Theorem \ref{thm4.5} follows from Proposition \ref{prop4.5} by the same
arguments as in the proof of Theorem \ref{thm2.1}.

\subsubsection{Functions with no boundary conditions}

Let $H^{1}\left( M\right) =W^{1,2}\left( M\right) $. It follows from \cite%
{CY1} that for every $u\in H^{1}\left( M\right) \backslash \left\{ 0\right\} 
$ with $\overline{u}=0$, we have%
\begin{equation}
\int_{M}e^{2\pi \frac{u^{2}}{\left\Vert \nabla u\right\Vert _{L^{2}}^{2}}%
}d\mu \leq c\left( M,g\right) .  \label{eq4.84}
\end{equation}%
As a consequence,%
\begin{equation}
\log \int_{M}e^{2u}d\mu \leq \frac{1}{2\pi }\left\Vert \nabla u\right\Vert
_{L^{2}}^{2}+c_{1}\left( M,g\right) .  \label{eq4.85}
\end{equation}

\begin{lemma}
\label{lem4.3}For any $u\in H^{1}\left( M\right) $ and $a>0$, we have%
\begin{equation}
\int_{M}e^{au^{2}}d\mu <\infty .  \label{eq4.86}
\end{equation}
\end{lemma}

Using (\ref{eq4.84}), the same argument as the proof of Lemma \ref{lem2.1}
implies Lemma \ref{lem4.3}.

\begin{proposition}
\label{prop4.6}Let $\left( M^{2},g\right) $ be a smooth compact Riemann
surface with nonempty boundary. Assume $u_{i}\in H^{1}\left( M\right) $ such
that $\overline{u_{i}}=0$, $u_{i}\rightharpoonup u$ weakly in $H^{1}\left(
M\right) $ and%
\begin{equation}
\left\vert \nabla u_{i}\right\vert ^{2}d\mu \rightarrow \left\vert \nabla
u\right\vert ^{2}d\mu +\sigma  \label{eq4.87}
\end{equation}%
as measure. Given $x\in M$ and $p\in \mathbb{R}$ such that $0<p<\frac{1}{%
\sigma \left( \left\{ x\right\} \right) }$.

\begin{enumerate}
\item If $x\in M\backslash \partial M$ , then for some $r>0$,%
\begin{equation}
\sup_{i}\int_{B_{r}\left( x\right) }e^{4\pi pu_{i}^{2}}d\mu <\infty .
\label{eq4.88}
\end{equation}

\item If $x\in \partial M$, then for some $r>0$,%
\begin{equation}
\sup_{i}\int_{B_{r}\left( x\right) }e^{2\pi pu_{i}^{2}}d\mu <\infty .
\label{eq4.89}
\end{equation}
\end{enumerate}
\end{proposition}

\begin{proof}
Fix $p_{1}\in \left( p,\frac{1}{\sigma \left( \left\{ x\right\} \right) }%
\right) $, then%
\begin{equation}
\sigma \left( \left\{ x\right\} \right) <\frac{1}{p_{1}}.  \label{eq4.90}
\end{equation}%
We can find a $\varepsilon >0$ such that%
\begin{equation}
\left( 1+\varepsilon \right) \sigma \left( \left\{ x\right\} \right) <\frac{1%
}{p_{1}}  \label{eq4.91}
\end{equation}%
and%
\begin{equation}
\left( 1+\varepsilon \right) p<p_{1}.  \label{eq4.92}
\end{equation}%
Let $v_{i}=u_{i}-u$, then $v_{i}\rightharpoonup 0$ weakly in $H^{1}\left(
M\right) $, $v_{i}\rightarrow 0$ in $L^{2}\left( M\right) $. For any $%
\varphi \in C^{\infty }\left( M\right) $, we have%
\begin{eqnarray*}
&&\left\Vert \nabla \left( \varphi v_{i}\right) \right\Vert _{L^{2}}^{2} \\
&\leq &\left( \left\Vert \varphi \nabla v_{i}\right\Vert _{L^{2}}+\left\Vert
v_{i}\nabla \varphi \right\Vert _{L^{2}}\right) ^{2} \\
&\leq &\left( \left\Vert \varphi \nabla u_{i}\right\Vert _{L^{2}}+\left\Vert
\varphi \nabla u\right\Vert _{L^{2}}+\left\Vert v_{i}\nabla \varphi
\right\Vert _{L^{2}}\right) ^{2} \\
&\leq &\left( 1+\varepsilon \right) \left\Vert \varphi \nabla
u_{i}\right\Vert _{L^{2}}^{2}+c\left( \varepsilon \right) \left\Vert \varphi
\nabla u\right\Vert _{L^{n}}^{n}+c\left( \varepsilon \right) \left\Vert
v_{i}\nabla \varphi \right\Vert _{L^{n}}^{n}.
\end{eqnarray*}%
It follows that%
\begin{eqnarray*}
&&\lim \sup_{i\rightarrow \infty }\left\Vert \nabla \left( \varphi
v_{i}\right) \right\Vert _{L^{2}}^{2} \\
&\leq &\left( 1+\varepsilon \right) \left( \int_{M}\varphi ^{2}d\sigma
+\int_{M}\varphi ^{2}\left\vert \nabla u\right\vert ^{2}d\mu \right)
+c\left( \varepsilon \right) \left\Vert \varphi \nabla u\right\Vert
_{L^{2}}^{2} \\
&=&\left( 1+\varepsilon \right) \int_{M}\varphi ^{2}d\sigma +c_{1}\left(
\varepsilon \right) \int_{M}\varphi ^{2}\left\vert \nabla u\right\vert
^{2}d\mu .
\end{eqnarray*}

If $x\in M\backslash \partial M$, then we can find a $\varphi \in
C_{c}^{\infty }\left( M\right) $ such that $\left. \varphi \right\vert
_{B_{r}\left( x\right) }=1$ for some $r>0$ and%
\begin{equation*}
\left( 1+\varepsilon \right) \int_{M}\varphi ^{2}d\sigma +c_{1}\left(
\varepsilon \right) \int_{M}\varphi ^{2}\left\vert \nabla u\right\vert
^{2}d\mu <\frac{1}{p_{1}}.
\end{equation*}%
Hence for $i$ large enough, we have%
\begin{equation*}
\left\Vert \nabla \left( \varphi v_{i}\right) \right\Vert _{L^{2}}^{2}<\frac{%
1}{p_{1}}.
\end{equation*}%
Note that $\varphi v_{i}\in H_{0}^{1}\left( M\right) $. We have%
\begin{eqnarray*}
\int_{B_{r}\left( x\right) }e^{4\pi p_{1}v_{i}^{2}}d\mu &\leq
&\int_{M}e^{4\pi p_{1}\left( \varphi v_{i}\right) ^{2}}d\mu \\
&\leq &\int_{M}e^{4\pi \frac{\left( \varphi v_{i}\right) ^{2}}{\left\Vert
\nabla \left( \varphi v_{i}\right) \right\Vert _{L^{2}}^{2}}}d\mu \\
&\leq &c\left( M,g\right) .
\end{eqnarray*}%
Next we observe that%
\begin{equation*}
u_{i}^{2}=\left( v_{i}+u\right) ^{2}\leq \left( 1+\varepsilon \right)
v_{i}^{2}+c\left( \varepsilon \right) u^{2},
\end{equation*}%
hence%
\begin{equation*}
e^{4\pi u_{i}^{2}}\leq e^{4\pi \left( 1+\varepsilon \right)
v_{i}^{2}}e^{c\left( \varepsilon \right) u^{2}}.
\end{equation*}%
Since $e^{4\pi \left( 1+\varepsilon \right) v_{i}^{2}}$ is bounded in $L^{%
\frac{p_{1}}{1+\varepsilon }}\left( B_{r}\left( x\right) \right) $, $%
e^{c\left( \varepsilon \right) u^{2}}\in L^{q}\left( B_{r}\left( x\right)
\right) $ for any $0<q<\infty $ (by Lemma \ref{lem4.2}), and $\frac{p_{1}}{%
1+\varepsilon }>p$, it follows from Holder inequality that $e^{4\pi
u_{i}^{2}}$ is bounded in $L^{p}\left( B_{r}\left( x\right) \right) $.

If $x\in \partial M$, then (\ref{eq4.84}) and similar arguments as in the
proof of Proposition \ref{prop2.1} tells us%
\begin{equation*}
\sup_{i}\int_{B_{r}\left( x\right) }e^{2\pi pu_{i}^{2}}d\mu <\infty .
\end{equation*}
\end{proof}

\begin{corollary}
\label{cor4.4}Assume $u_{i}\in H^{1}\left( M\right) $ such that $\overline{%
u_{i}}=0$ and $\left\Vert \nabla u_{i}\right\Vert _{L^{2}}\leq 1$. We also
assume $u_{i}\rightharpoonup u$ weakly in $H^{1}\left( M\right) $ and%
\begin{equation}
\left\vert \nabla u_{i}\right\vert ^{2}d\mu \rightarrow \left\vert \nabla
u\right\vert ^{2}d\mu +\sigma  \label{eq4.93}
\end{equation}%
as measure. Let $K$ be a compact subset of $M$ and%
\begin{eqnarray}
\kappa _{0} &=&\max_{x\in K\backslash \partial M}\sigma \left( \left\{
x\right\} \right) ;  \label{eq4.94} \\
\kappa _{1} &=&\max_{x\in K\cap \partial M}\sigma \left( \left\{ x\right\}
\right) ;  \label{eq4.95}
\end{eqnarray}

\begin{enumerate}
\item If $\kappa _{1}<1$, then for any $1\leq p<\min \left\{ \frac{2}{\kappa
_{0}},\frac{1}{\kappa _{1}}\right\} $, 
\begin{equation}
\sup_{i}\int_{K}e^{2\pi pu_{i}^{2}}d\mu <\infty .  \label{eq4.96}
\end{equation}

\item If $\kappa _{1}=1$, then $\sigma =\delta _{x_{0}}$ for some $x_{0}\in
\partial M$, $u=0$ and after passing to a subsequence,%
\begin{equation}
e^{2\pi u_{i}^{2}}\rightarrow 1+c_{0}\delta _{x_{0}}  \label{eq4.97}
\end{equation}%
as measure for some $c_{0}\geq 0$.
\end{enumerate}
\end{corollary}

\begin{proof}
First assume $\kappa _{1}<1$. We claim for any $x\in K$, there exists a $%
r_{x}>0$ such that%
\begin{equation*}
\sup_{i}\int_{B_{r_{x}}\left( x\right) }e^{2\pi pu_{i}^{2}}d\mu <\infty .
\end{equation*}%
Once this claim is proved, we can deduce%
\begin{equation*}
\sup_{i}\int_{K}e^{2\pi pu_{i}^{2}}d\mu <\infty
\end{equation*}%
by the covering argument in the proof of Corollary \ref{cor2.1}.

If $x\in K\backslash \partial M$, then%
\begin{equation*}
p<\frac{2}{\kappa _{0}}\leq \frac{2}{\sigma \left( \left\{ x\right\} \right) 
}.
\end{equation*}%
Hence $\frac{p}{2}<\frac{1}{\sigma \left( \left\{ x\right\} \right) }$. It
follows from Proposition \ref{prop4.6} that for some $r>0$,%
\begin{equation*}
\sup_{i}\int_{B_{r}\left( x\right) }e^{2\pi pu_{i}^{2}}d\mu
=\sup_{i}\int_{B_{r}\left( x\right) }e^{4\pi \cdot \frac{p}{2}u_{i}^{2}}d\mu
<\infty .
\end{equation*}

If $x\in K\cap \partial M$, then%
\begin{equation*}
p<\frac{1}{\kappa _{1}}\leq \frac{1}{\sigma \left( \left\{ x\right\} \right) 
}.
\end{equation*}%
It follows from Proposition \ref{prop4.6} again that for some $r>0$,%
\begin{equation*}
\sup_{i}\int_{B_{r}\left( x\right) }e^{2\pi pu_{i}^{2}}d\mu <\infty .
\end{equation*}

This proves the claim.

The case when $\kappa _{1}=1$ can be handled the same way as in the proof of
Corollary \ref{cor2.1}.
\end{proof}

\begin{theorem}
\label{thm4.6}Let $\left( M^{2},g\right) $ be a smooth compact Riemann
surface with nonempty boundary. Assume $\alpha >0$, $m_{i}>0$, $%
m_{i}\rightarrow \infty $, $u_{i}\in H^{1}\left( M\right) $ such that $%
\overline{u_{i}}=0$ and%
\begin{equation}
\log \int_{M}e^{2m_{i}u_{i}}d\mu \geq \alpha m_{i}^{2}.  \label{eq4.98}
\end{equation}%
We also assume $u_{i}\rightharpoonup u$ weakly in $H^{1}\left( M\right) $, $%
\left\vert \nabla u_{i}\right\vert ^{2}d\mu \rightarrow \left\vert \nabla
u\right\vert ^{2}d\mu +\sigma $ as measure and%
\begin{equation}
\frac{e^{2m_{i}u_{i}}}{\int_{M}e^{2m_{i}u_{i}}d\mu }\rightarrow \nu
\label{eq4.99}
\end{equation}%
as measure. Let%
\begin{eqnarray}
&&\left\{ x\in M\backslash \partial M:\sigma \left( \left\{ x\right\}
\right) \geq 4\pi \alpha \right\} \cup \left\{ x\in \partial M:\sigma \left(
\left\{ x\right\} \right) \geq 2\pi \alpha \right\}  \label{eq4.100} \\
&=&\left\{ x_{1},\cdots ,x_{N}\right\} ,  \notag
\end{eqnarray}%
then%
\begin{equation}
\nu =\sum_{i=1}^{N}\nu _{i}\delta _{x_{i}},  \label{eq4.101}
\end{equation}%
here $\nu _{i}\geq 0$ and $\sum_{i=1}^{N}\nu _{i}=1$.
\end{theorem}

\begin{proof}
Assume $x\in M\backslash \partial M$ with $\sigma \left( \left\{ x\right\}
\right) <4\pi \alpha $, then we claim that for some $r>0$, $\nu \left(
B_{r}\left( x\right) \right) =0$. Indeed we fix $p$ such that%
\begin{equation}
\frac{1}{4\pi \alpha }<p<\frac{1}{\sigma \left( \left\{ x\right\} \right) },
\label{eq4.102}
\end{equation}%
it follows from Proposition \ref{prop4.6} that for some $r>0$,%
\begin{equation}
\int_{B_{r}\left( x\right) }e^{4\pi pu_{i}^{2}}d\mu \leq c,  \label{eq4.103}
\end{equation}%
here $c$ is a positive constant independent of $i$. We have%
\begin{equation}
2m_{i}u_{i}\leq 4\pi pu_{i}^{2}+\frac{m_{i}^{2}}{4\pi p},  \label{eq4.104}
\end{equation}%
hence%
\begin{equation*}
\int_{B_{r}\left( x\right) }e^{2m_{i}u_{i}}d\mu \leq ce^{\frac{m_{i}^{2}}{%
4\pi p}}.
\end{equation*}%
It follows that%
\begin{equation*}
\frac{\int_{B_{r}\left( x\right) }e^{2m_{i}u_{i}}d\mu }{%
\int_{M}e^{2m_{i}u_{i}}d\mu }\leq ce^{\left( \frac{1}{4\pi p}-\alpha \right)
m_{i}^{2}}.
\end{equation*}%
In particular,%
\begin{equation*}
\nu \left( B_{r}\left( x\right) \right) \leq \lim \inf_{i\rightarrow \infty }%
\frac{\int_{B_{r}\left( x\right) }e^{2m_{i}u_{i}}d\mu }{%
\int_{M}e^{2m_{i}u_{i}}d\mu }=0.
\end{equation*}%
We get $\nu \left( B_{r}\left( x\right) \right) =0$.

If $x\in \partial M$ with $\sigma \left( \left\{ x\right\} \right) <2\pi
\alpha $, then similar argument shows for some $r>0$, $\nu \left(
B_{r}\left( x\right) \right) =0$.

Clearly these imply%
\begin{equation}
\nu \left( M\backslash \left\{ x_{1},\cdots ,x_{N}\right\} \right) =0.
\label{eq4.105}
\end{equation}%
Hence $\nu =\sum_{i=1}^{N}\nu _{i}\delta _{x_{i}},$ with $\nu _{i}\geq 0$
and $\sum_{i=1}^{N}\nu _{i}=1$.
\end{proof}

\end{document}